\input ./style/arxiv-general.cfg
\documentclass[MSNbibl,number,citesort,seceqn,dvips]{arxbj}
\makeatletter
   \@ifpackageloaded{graphicx}{}{\usepackage{graphicx}}
\makeatother
\usepackage{accents,upgreek}

\aid{0}
\volume{21}
\issue{3}
\pubyear{2015}
\firstpage{1538}
\lastpage{1574}
\doi{10.3150/14-BEJ613} 

\makeatletter

\newcommand{\rright}{\right}
\newcommand{\lleft}{\left}
\newcommand{\rrVert}{\Vert}
\newcommand{\rrvert}{\vert}
\newcommand{\llVert}{\Vert}
\newcommand{\llvert}{\vert}
\newtheorem{theorem}{Theorem}[section]
\newtheorem{lemma}[theorem]{Lemma}
\newremark{remark}[theorem]{Remark}

\newcommand{\us}{\underline{s}}
\newcommand{\E}{\mathbb{E}}
\newcommand{\var}{\operatorname{\mathbb{V}ar}}
\newcommand{\cov}{\operatorname{\mathbb{C}ov}}
\renewcommand{\P}{\mathbb{P}}
\newcommand{\ep}{\varepsilon}
\newcommand{\De}{\Delta}
\newcommand{\la}{\lambda}
\renewcommand{\th}{\theta}
\newcommand{\ga}{\gamma}
\newcommand{\pmu}{\boldsymbol{\mu}}
\newcommand{\Si}{\boldsymbol{\Sigma}}
\newcommand{\bbC}{\mathbb{C}}
\newcommand{\bbR}{\mathbb{R}}
\newcommand{\cC}{\mathcal{C}}
\newcommand{\cD}{\mathcal{D}}
\newcommand{\cK}{\mathcal{K}}
\newcommand{\cQ}{\mathcal{Q}}

\newcommand{\frm}{\mathfrak{m}}
\newcommand{\frt}{\mathfrak{t}}
\newcommand{\cR}{\mathcal{R}}
\newcommand{\cT}{\mathcal{T}}
\newcommand{\vp}{\varpi}

\newcommand{\A}{\mathbf{A}}
\newcommand{\C}{\mathbf{C}}
\newcommand{\B}{\mathbf{B}}
\newcommand{\D}{\mathbf{D}}
\newcommand{\F}{\mathbf{F}}
\newcommand{\G}{\mathbf{G}}
\renewcommand{\H}{\mathbf{H}}
\newcommand{\I}{\mathbf{I}}
\newcommand{\M}{\mathbf{M}}
\newcommand{\N}{\mathbf{N}}
\renewcommand{\r}{\mathbf{r}}
\renewcommand{\S}{\mathbf{S}}
\newcommand{\T}{\mathbf{T}}
\newcommand{\X}{\mathbf{X}}
\newcommand{\Z}{\mathbf{Z}}
\newcommand{\bz}{\mathbf{z}}

\DeclareMathAlphabet{\mathbbl}{U}{mt2hrb}{m}{n}
\SetMathAlphabet{\mathbbl}{bold}{U}{mt2hrb}{b}{n}

\newcommand{\eqref}[1]{(\ref{#1})}

\newcommand{\tr}{\operatorname{tr}}

\def\sfrac#1#2{#1/#2}

\def\afrac#1#2{#1/(#2)}

\def\sklfrac#1#2{(#1/#2)}
\def\sklvfrac#1#2{((#1)/#2)}
\def\sklafrac#1#2{(#1/(#2))}

\makeatother

\begin{document}
\begin{frontmatter}

\title{Convergence of the empirical spectral distribution function of
Beta matrices}
\runtitle{Convergence of ESDF of Beta matrices}

\begin{aug}
\author[1]{\inits{Z.}\fnms{Zhidong} \snm{Bai}\thanksref{1,e1}\ead[label=e1,mark]{baizd@nenu.edu.cn}},
\author[1]{\inits{J.}\fnms{Jiang} \snm{Hu}\corref{}\thanksref{1,e2}\ead[label=e2,mark]{huj156@nenu.edu.cn}},
\author[2]{\inits{G.}\fnms{Guangming} \snm{Pan}\thanksref{2}\ead[label=e3]{gmpan@ntu.edu.sg}}
\and
\author[3]{\inits{W.}\fnms{Wang} \snm{Zhou}\thanksref{3}\ead[label=e4]{stazw@nus.edu.sg}}
\address[1]{KLASMOE and School of Mathematics $\&$ Statistics,
Northeast Normal University, Changchun, P.R.C., 130024.
\printead{e1,e2}}

\address[2]{Division of Mathematical Sciences, School of Physical and
Mathematical Sciences, Nanyang Technological University, Singapore 637371.
\printead{e3}}

\address[3]{Department of Statistics and Applied Probability, National
University of
Singapore, Singapore 117546.
\printead{e4}}
\end{aug}

\received{\smonth{8} \syear{2012}}
\revised{\smonth{11} \syear{2013}}

%
\begin{abstract}
Let $\mathbf{B}_n=\S_n (\S_n+\alpha_n\T_N )^{-1}$,
where $\S_n$ and $\T_N$ are two independent sample covariance
matrices with dimension $p$ and sample sizes $n$ and $N$, respectively.
This is the so-called Beta matrix. In this paper,
we focus on the limiting spectral distribution function and the central
limit theorem of linear
spectral statistics of $\B_n$. Especially, we do not require $\S_n$
or $\T_N$ to be invertible. Namely, we can
deal with the case where $p > \max\{n,N\}$ and $p<n+N$.
Therefore, our results cover many important applications which cannot
be simply deduced from the corresponding results for multivariate $F$ matrices.
\end{abstract}

%

\end{frontmatter}

\section{Introduction}

In the last two decades, more and more large dimensional data sets
appear in scientific research. When the dimension of data or number of
parameters becomes large, the classical methods could reduce
statistical efficiency significantly. In order to analyze those large
data sets,
many new statistical techniques, such as large dimensional multivariate
statistical analysis (MSA) based on the random matrix theory (RMT),
have been developed. In this paper, we will investigate a widely used
type of random matrices in MSA which are called Beta matrices.

Firstly we introduce some definitions and terminology associated with
Beta matrices. Let $\mathbf{X}_n=(x_{ij})_{p\times n}$, where $\{
x_{ij}\}$ are independent and identically distributed (i.i.d.) random
variables with mean zero and variance one, and
similarly let $\T_N=N^{-1}\mathbb{X}_N\mathbb{X}_N^*$ be another
sample covariance matrix independent of $\S_n$, where $\mathbb
{X}_N=(\mathbbl{x}_{ij})_{p\times N}$ and $\{\mathbbl{x}_{ij}\}$ are
i.i.d. random variables with mean zero and variance one.
The Beta matrix is defined as
%
\begin{eqnarray}\label{bema}
\mathbf{B}_n=\S_n (\S_n+\alpha_n
\T_N )^{-1},
\end{eqnarray}
where $\alpha_n$ is a positive constant.
For any $n\times n$ matrix $\mathbf{A}$ with only real eigenvalues, we
denote $F^{\mathbf{A}}$ as the empirical spectral distribution
function (ESDF) of $\mathbf{A}$, that is
$F^\mathbf{A}(x)=\frac{1}{n}\sum_{i=1}^nI(\lambda_i^{\mathbf
{A}}\leq x)$, where $\lambda_i^{\mathbf{A}}$ denotes the $i$th
smallest eigenvalue of $\mathbf{A}$ and $I(\cdot)$ is the indicator
function. In addition, we shall call
$
\int f(x)\,\mathrm{d}F^\A(x)=\frac{1}{n}\sum_{k=1}^nf(\la^\A_k)
$
a linear
spectral statistics (LSS) of matrix $\A$.
In this paper, we focus on the limiting ESDF and the central limit
theorem (CLT) of LSS of $\B_n$.

One motivation to study Beta matrices is that their ESDFs are very
useful in MSA, such as in the test of equality of $k\ (k \geq2)$
covariance matrices,
multivariate analysis of variance, the independence test of
sets of variables, canonical correlation analysis and so on. There is a
huge literature regarding this kind of matrices. One may refer to \cite
{Anderson03I,FujikoshiU11M,Muirhead82A} for more details. For
pedagogical reasons, we provide one statistical application of Beta
matrices as follows.

Let $\{\bz_{1}^{(1)},\dots,\bz_{n}^{(1)}\}$ be an i.i.d. sample
drawn from a $p$-dimensional distribution and $\{\bz_{1}^{(2)},\dots
,\allowbreak \bz_{N}^{(2)}\}$ be an i.i.d. sample drawn from another
$p$-dimensional distribution. Suppose $\pmu_{i}=\E\bz
_{1}^{(i)}=\mathbf{0}$ and $\Si_{i}=\var\bz_{1}^{(i)}$, $i=1,2$. Write
$\bz_j^{(1)}=\Si_1^{1/2}\X_{(\cdot,j)}$ and $\bz_j^{(2)}=\Si
_1^{1/2}\mathbb{X}_{(\cdot,j)}$ where $\X_{(\cdot,j)}\ ( \mathbb
{X}_{(\cdot,j)} )$ is the $j$th column of $\X_n\ (\mathbb{X}_N)$
and $\Si_i^{1/2}$ is any square root of $\Si_i$. We wish to test
\begin{eqnarray*}
H_0\dvtx \Si_1=\Si_2 \quad \mbox{v.s.}\quad
H_1\dvtx \Si_1\neq\Si_2.
\end{eqnarray*}
This is one of the most elementary problems in MSA, for which there are
lots of test statistics. If we write $\Z_n^{(1)}=n^{-1}\sum_{i=1}^n\bz_{i}^{(1)}(\bz_{i}^{(1)})^*$ and $\Z_N^{(2)}=N^{-1}\sum_{j=1}^n\bz_{j}^{(2)}(\bz_{j}^{(2)})^*$, then all the following
$L_j$, $j=1,2,\ldots,5$ are the most frequently used test statistics
for $H_0$ (see Chapter~8 in \cite{Muirhead82A}).
%
\begin{eqnarray}
\label{lsss} L_1&=&\log\frac{|\Z_n^{(1)}|^{n}\cdot|\Z_N^{(2)}|^{N}}{|c_n\Z
_n^{(1)}+c_N\Z_N^{(2)}|^{n+N}}=\int \bigl(n
\log({x}/{c_n})- N\log \bigl((1-x)/c_N\bigr) \bigr) \,
\mathrm{d}F^{\B_n}(x),\nonumber
\\
L_2&=&\tr\bigl(\Z_N^{(2)}\bigl(
\Z_n^{(1)}\bigr)^{-1}\bigr)=p\int\frac{1-x}{\alpha_nx}
\,\mathrm{d}F^{\B_n}(x),\nonumber
\\
L_3&=&\log\bigl\llvert \Z_n^{(1)}\bigl(
\Z_n^{(1)}+\alpha_n\Z _N^{(2)}
\bigr)^{-1}\bigr\rrvert =p\int\log x \,\mathrm{d}F^{\B_n}(x),
\nonumber\\[-8pt]\\[-8pt]
L_4&=&\tr\bigl(\Z_n^{(1)}\bigl(
\Z_n^{(1)}+\alpha_n\Z_N^{(2)}
\bigr)^{-1}\bigr)=p\int x \,\mathrm{d}F^{\B_n}(x),
\nonumber\\
L_5&=&c_n\tr\bigl(\Z_n^{(1)}
\bigl(c_n\Z_n^{(1)}+c_N
\Z_N^{(2)}\bigr)^{-1}-\I \bigr)^2+c_N\tr
\bigl(\Z_N^{(2)}\bigl(c_n\Z_n^{(1)}+c_N
\Z_N^{(2)}\bigr)^{-1}-\I\bigr)^2\nonumber
\\
&=&c_np\int\bigl(c_n^{-1}x-1
\bigr)^2 \,\mathrm{d}F^{\B_n}(x)+c_{N}p\int
\bigl(c_N^{-1}(1-x)-1\bigr)^2 \,
\mathrm{d}F^{\B_n}(x),\nonumber %
\end{eqnarray}
where
$c_n=n/(n+N)$, $c_N=N/(n+N)$ and $\alpha_n=N/n$.
Apparently all the above test statistics are linear functionals of the
ESDF of Beta
matrices $\B_n$, which are all the LSS of $\B_n$.
It is already well known that the classical limit theorems for those
LSS are not valid when the dimension is large.
So it is crucial to investigate the sequence $\{F^{\B_n}\}$ in the
large dimensional case. The following result tells us the limiting
behavior of $\{F^{\B_n}\}$ as $p,n,N\to\infty$.
%
\begin{theorem}[(Limiting spectral distribution function (LSDF))]\label{lsd}
Assume on a common probability space:
\begin{enumerate}[(iii)]
\item[(i)] For each $i,j,n$, $x_{ij}=x_{nij}$ are i.i.d.
with $\E x_{11}=0$, $\E| x_{11}|^2=1$.
%
\item[(ii)]$\alpha_n\to\alpha>0$ and $y_n=p/n\to y >0$.
\item[(iii)] For each $k,l,N$, $\mathbbl{x}_{kl}=\mathbbl
{x}_{Nkl}$ are i.i.d. with $\E\mathbbl{x}_{11}=0$, $\E| \mathbbl
{x}_{11}|^2=1$.
%
\item[(iv)] $Y_N=p/N\to Y>0$ and $\frac{p}{n+N}\to
\frac{yY}{y+Y}\in(0,1)$.
\item[(v)] $\sup_n\E|x_{11}|^4<\infty$ and $\sup_N\E
|\mathbbl x_{11}|^4<\infty$.
\end{enumerate}
Then with probability $1$, $F^{\B_n}{\rightarrow} F$ weakly, where $F$
is a non-random distribution function whose density function is
\begin{eqnarray*}
\lleft\{ %
\begin{array} {l@{\qquad}l} \displaystyle \frac{\sqrt{((\alpha(1-Y)-1+y)^2+4\alpha)(t_r-t)(t-t_l)}}{2 \uppi
t(1-t)(y(1-t)+\alpha tY)}, & \mbox{when
$t_l<t<t_r$;}
\\
0, & \mbox{otherwise,} \end{array} %
\rright.
\end{eqnarray*}
where $t_l,t_r= (\frac{2\alpha-(1-y)[\alpha(1-Y)-1+y]\mp
2\alpha\sqrt{y-yY+Y}}{(\alpha(1-Y)-1+y)^2+4\alpha} )$. In addition,
when $y>1$, $F(t)$ has a point mass $1-1/y$ at $t=0$; when $Y>1$,
$F(t)$ has a point mass $1-1/Y$ at $t=1$.
\end{theorem}
%
\begin{remark}\label{remark1.2}
Condition $yY/(y+Y)<1$ is to guarantee that the random matrix $\S
_n+\alpha_n\T_N$ is invertible almost surely because $yY/(y+Y)>1$
ensures that the dimension $p$ could be eventually larger than the
number of observations $n+N$. This would imply that $\S_n+\alpha_n\T
_N$ is singular.
Condition (v) gives us the a.s. bounds of the limit of
the smallest and largest eigenvalues, $\la_1^{\S_n+\alpha_n\T_N}$
and $\la_p^{\S_n+\alpha_n\T_N}$ respectively, of the random matrix
$\S_n+\alpha_n\T_N$ since by the definition of $\B_n$ we can rewrite
\begin{eqnarray*}
&&\S_n+\alpha_n\T_N\\
&&\quad =\frac{1}{n} \biggl(
\X_{i}\X_{i}^*+\frac{\alpha
_n n}{N}\mathbb{X}_N
\mathbb{X}_N^* \biggr)
\\
&&\quad =\frac{1}{n+N}\lleft( %
\begin{array} {c@{\quad}c@{\quad}c@{\quad}c@{\quad}c@{\quad}c} x_{11} &
\cdots& x_{1n} & \mathbbl{x}_{11} & \cdots& \mathbbl
{x}_{1N}
\\
\vdots& \vdots& \vdots& \vdots& \vdots& \vdots
\\
x_{p1} & \cdots& x_{pn} & \mathbbl{x}_{p1} &
\cdots& \mathbbl {x}_{pN}
\\
\end{array} %
\rright)\boldsymbol{\Gamma}\lleft( %
\begin{array}
{c@{\quad}c@{\quad}c@{\quad}c@{\quad}c@{\quad}c} x_{11} & \cdots& x_{1n} & \mathbbl{x}_{11}
& \cdots& \mathbbl {x}_{1N}
\\
\vdots& \vdots& \vdots& \vdots& \vdots& \vdots
\\
x_{p1} & \cdots& x_{pn} & \mathbbl{x}_{p1} &
\cdots& \mathbbl {x}_{pN}
\\
\end{array} %
\rright)^*.
\end{eqnarray*}
Here
\begin{eqnarray*}
\boldsymbol{\Gamma}=\lleft( %
\begin{array} {c@{\quad}c} \left( %
\begin{array} {c@{\quad}c@{\quad}c} \displaystyle \frac{n+N}{n}&&
\\
&\ddots&
\\
&&\displaystyle \frac{n+N}{n}
\\
\end{array} %
\right)_{n\times n}&
\\
& \left( %
\begin{array} {c@{\quad}c@{\quad}c} \displaystyle \frac{(n+N)\alpha_n}{N}&&
\\
&\ddots&
\\
&&\displaystyle \frac{(n+N)\alpha_n}{N}
\\
\end{array} %
\right)_{N\times N} \end{array} %
\rright)_{(n+N)\times(n+N) }
\end{eqnarray*}
is a diagonal matrix. Thus under (v), for any $\ep>0$
and any $l>0$, there exist two positive constants $\nu_{1}=\min\{
1,\alpha Y/y\}\cdot(1+y/Y) (1-\sqrt{\frac{yY}{y+Y}} )^2$
and $\nu_{2}=\max\{1,\alpha Y/y\}\cdot(1+y/Y) (1+\sqrt{\frac
{yY}{y+Y}} )^2$ such that almost surely
%
\begin{eqnarray}
\label{rem1.2} \lim_{p,n,N\to\infty} \la_1^{\S_n+\alpha_n\T_N}\geq
\nu_1, \qquad \lim_{p,n,N\to\infty} \la_p^{\S_n+\alpha_n\T_N}
\leq\nu_2
\end{eqnarray}
and
%
\begin{eqnarray}
\label{rem1.3} \P \bigl(\la_1^{\S_n+\alpha_n\T_N}<\nu_1-\ep
\bigr)=\mathrm{o}\bigl(n^{-l}\bigr),\qquad  \P \bigl(\la_p^{\S_n+\alpha_n\T_N}>
\nu_2+\ep \bigr)=\mathrm{o}\bigl(n^{-l}\bigr).
\end{eqnarray}
One may refer to \cite{BaiS10S} for the proof of \eqref{rem1.2} and
\eqref{rem1.3}.
\end{remark}
%
\begin{remark}
Under the assumptions (i) and (ii) in
Theorem~\ref{lsd}, it is proved that
the ESDF of the sequence
$\{\S_n\}$ has a non-random limit which is known as the
Marchenko--Pastur (M--P) distribution \cite{MarvcenkoP67D,BaiS10S}. Yin
\cite{Yin86L} and Silverstein \cite{Silverstein95S} investigated the
LSDF of the sequence
$\{\S_n\T_N\}$ assuming (i)--(iii) of
Theorem~\ref{lsd}.
If $\T_N$ is invertible, Bai \textit{et al.} \cite{BaiY87L} gave the LSDF of
the sequence
$\{\S_n\T_N^{-1}\}$.
\end{remark}
%
\begin{remark}
If $\max\{y,Y\}<1$, by (v) we know that at least one of
the matrices $\S_n$ and $\T_N$ is invertible a.s. Without loss of
generality, we assume $Y<1$. So $\T_N$ is invertible a.s. Then we have
%
\begin{eqnarray}\label{betanfn}
\B_n=\S_n\T_N^{-1} \bigl(
\S_n\T_N^{-1}+\alpha_n\I
\bigr)^{-1},
\end{eqnarray}
which is a function of $\S_n\T_N^{-1}$. Via $\tilde t=\alpha_n
t/(1-t)$ we can recover Theorem~5.3 in \cite{BaiY87L} from our Theorem~\ref{lsd} directly.
Thus our Theorem~\ref{lsd} includes Theorem~5.3 in \cite{BaiY87L} as
a special case.
\end{remark}
%
\begin{remark}
From the density function in Theorem~\ref{lsd}, we can find that the
condition $\frac{p}{n+N}\to\frac{yY}{y+Y}\in(0,1)$ is necessary,
which is to make sense of $\sqrt{y+Y-yY}$.
\end{remark}

For the purpose of multivariate inference, it is of interest to know
the limiting distribution of these LSS \eqref{lsss}.
Thus, we will give the central limit theorems (CLT) of LSS of Beta
matrices. In order to present this result, we need more notation.
Denote
\begin{eqnarray*}
\mathfrak{B}_n(x)=p\bigl(F^{\B_n}(x)-F_0(x)
\bigr),
\end{eqnarray*}
where $F_0$ is the limit distribution of $F^{\B_n}$ with $\alpha,y,Y$
replaced by $\alpha_n,y_n,Y_N$, respectively.
For any function of bounded variation $G$ on the real line, its
Stieltjes transform is defined by
\[
s_G(z)=\int\frac{1}{\lambda-z}\,\mathrm{d}G(\lambda),\qquad  z\in\mathbb
{C}^{+}\equiv\{z\in\mathbb{C}\dvtx \Im z>0\}.
\]
Then we have the following theorem.
%
\begin{theorem}\label{clt}
In addition to the conditions \textup{(i)}--\textup{(iv)} in
Theorem~\ref{lsd}, we further assume
that:
\begin{enumerate}[2.]

\item[1.]$\E x_{11}^2=\E\mathbbl x_{11}^2=\frt$, $\E
|x_{11}|^4=\mathfrak{m}_x$, $\E|\mathbbl x_{11}|^4=\mathfrak
{m}_\mathbbl{x}$ and $\max_{p,n,N}\{\mathfrak{m}_x,\mathfrak
{m}_\mathbbl{x}\}<\infty$, where $\frt=0$, when both $\X_n$ and
$\mathbb{X}_N$ are complex valued, and $\frt=1$ if both real.
\item[2.] Let $f_1,\dots,f_k$ be functions analytic on an open region
containing the interval $[c_l,c_r]$ where $c_l=\nu_2^{-1}(1-\sqrt {y})^2$, $c_r=1-\alpha\nu_2^{-1}(1-\sqrt{Y})^2$, and $\nu_2$ is
defined in Remark~\ref{remark1.2}.
\end{enumerate}
Then, as $\min(n,N,p)\to\infty$, the random vector
\begin{eqnarray*}
\biggl(\int f_i\,\mathrm{d}\mathfrak{B}_n(x) \biggr),\qquad i=1,\dots,k,
\end{eqnarray*}
converges weakly to a Gaussian vector $(G_{f_1},\dots,G_{f_k})$ with
mean functions
\begin{eqnarray*}
\E G_{f_i}&=&\frac{\frt}{4\uppi i}\oint f_i\biggl(
\frac{z}{\alpha+z}\biggr) \,\mathrm{d}\log \biggl(\frac{(1-Y)\dddot s^2(z)+2\dddot s(z)+1-y}{(1-Y)\dddot
s^2+2\dddot s(z)+1} \biggr)
\\
&&{}+\frac{\frt}{4\uppi i}\oint f_i\biggl(\frac{z}{\alpha+z}\biggr) \,
\mathrm{d}\log \bigl(1-Y\dddot s^2(z) \bigl(1+\dddot s(z)
\bigr)^{-2} \bigr)
\\
&&{}+\frac{\frm_x-\frt-2}{2\uppi i}\oint yf_i\biggl(\frac{z}{\alpha
+z}\biggr)
\bigl(\dddot s(z)+1\bigr)^{-3} \,\mathrm{d}\dddot s(z)
\\
&&{}+\frac{\frm_\mathbbl{x}-\frt-2}{4\uppi i}\oint f_i\biggl(\frac{z}{\alpha
+z}\biggr)
\bigl(1-Y\dddot s^2(z) \bigl(1+\dddot s(z)\bigr)^{-2}\bigr)
\,\mathrm{d}\log\bigl(1-Y\dddot s^2(z) \bigl(1+\dddot s(z)
\bigr)^{-2}\bigr)
\end{eqnarray*}
and covariance functions
\begin{eqnarray*}
&&\cov(G_{f_i},G_{f_j})\\
&&\quad =-\frac{\frt+1}{4\uppi^2}\oint\oint
\frac
{f_i(\afrac{z_1}{\alpha+z_1})f_j(\afrac{z_2}{\alpha+z_2}) \,\mathrm{d}\dddot
s(z_1)\,\mathrm{d}\dddot s(z_2)}{(\dddot s(z_1)-\dddot s(z_2))^2}
\\
&&\qquad {}-\frac{y(\frm_x-\frt-2)+Y(\frm_\mathbbl{x}-\frt-2)}{4\uppi^2}\oint \oint\frac{f_i(\afrac{z_1}{\alpha+z_1})f_j(\afrac{z_2}{\alpha+z_2})
\,\mathrm{d}\dddot s(z_1)\,\mathrm{d}\dddot s(z_2)}{(\dddot s(z_1)+1)^2(\dddot s(z_2)+1)^2},
\end{eqnarray*}
where
\begin{eqnarray*}
s(z)&=& \frac{(1+y)(1-z)-\alpha z(1-Y)+\sqrt{((1-y)(1-z)+\alpha
z(1-Y))^2-4\alpha z(1-z)}}{2 z(1-z)(y(1-z)+\alpha zY)}- \frac{1}{ z},
\\
\dot s(z)&=&\frac{\alpha}{(\alpha+z)^2} s\biggl(\frac{z}{\alpha+z}\biggr)-
\frac
{1}{\alpha+z}, \qquad \ddot s(z)=-z^{-1}(1-y)+y\dot s(z),
\\
s_{\mathrm{mp}}^Y(z)&=&\frac{1-Y-z+\sqrt{(z-1-Y)^2-4Y}}{2Yz},\qquad  \dddot
s(z)=Ys_{\mathrm{mp}}^Y\bigl(- \ddot s(z)\bigr)+\bigl(\ddot s(z)
\bigr)^{-1}(1-Y).
\end{eqnarray*}
%
All the above contour integrals can be evaluated on any contour
enclosing the interval $[\frac{\alpha c_l}{1- c_l},\frac{\alpha
c_r}{1- c_r}]$.
\end{theorem}
%
\begin{remark}
Actually, this result should be right under the condition that $f_i$ is
analytic (or continuously differentiable) on an open region containing
the interval $[t_l,t_r]$. However its proof is more difficult at the
current stage because we do not have the following results of Beta
matrices: the exact separation of eigenvalues, the limit of the
smallest and the largest eigenvalues and the convergence rate of the
ESDF.
\end{remark}
%
\begin{remark}
In this theorem, the notions $s(z)$ and $s^Y_{\mathrm{mp}}(z)$ are the Stieltjes
transforms of the LSDFs of $\B_n$ and $\T_N$ respectively. If $Y<1$,
Zheng in \cite{Zheng12C} established the CLT of the LSS of $F$ matrix
$\S_n\T_N^{-1}$ whose proof is based on \cite{BaiS04C}. It is
apparent that our Theorem~\ref{clt} covers Zheng's result. In
addition, notice that the conclusions in Theorem~\ref{clt} and Theorem~4.1 in \cite{Zheng12C} have the same form. The reason is that, by
calculation we can easily get
\begin{eqnarray*}
\dot s(z)=\frac{1}{zy}-\frac{1}{z}-\frac{y(z(1-Y)+1-y)+2zY-y\sqrt {((1-y)+z(1-Y))^2-4z}}{2z(yz+Y)}
\end{eqnarray*}
which has the same expression of the Stieltjes transform of the LSDF of
$F$ matrices (see (2.6) in \cite{Zheng12C}). Here we want to remind
the reader that, when we use the last formula to calculate the density
function, that is, calculating $\uppi^{-1}\lim_{z\downarrow x+i0}\Im
\dot s(z)$, we can find that the condition $Y<1$ is not needed but
$y+Y>yY$ is necessary (see page 79 in \cite{BaiS10S} for more details).
\end{remark}
%
\begin{remark}
If $\{x_{ij}\}$ and $\{\mathbbl{x}_{ij}\}$ are independent standard
normal random variables and $p<\max\{n,N\}$, Beta matrices can be seen
as Beta--Jacobi ensemble with some parameter $\beta$. Some related
results about this ensemble can be found in \cite{DumitriuP12G} and
the references therein.
\end{remark}

This paper is organized as follows: In Section~\ref{sec2}, we present the proof
of Theorem~\ref{lsd}. Theorem~\ref{clt} is proved in Section~\ref{sec3} and
Section~\ref{sec4}. Some technical lemmas are given
in Section~\ref{sec5}.

\section{Proof of Theorem \texorpdfstring{\protect\ref{lsd}}{1.1}}\label{sec2}

In this section, we will give the proof of Theorem~\ref{lsd}. The main
tool we use here is the Stieltjes transform. Its function can be
explained by the following two lemmas.
%
\begin{lemma}[(Lemma~1.1 in \cite{BaiZ10L})]\label{lebaiz}
For any random matrix $\A_n$, let $F^{\A_n}$ denote the ESDF of $\A
_n$ and $s_{F^{\A_n}}(z)$ its Stieltjes transform. Then, if $F^{\A
_n}$ is tight with probability one and for each $z\in\mathbb{C}^+$,
$s_{F^{\A_n}}(z)$ converges almost surely to a non-random limit
$s_F(z)$ as $n \to\infty$, then there exists a non-random probability
distribution $F$ taking $s_F(z)$ as its
Stieltjes transform such that with probability one, as $n\to\infty$,
$F^{\A_n}$ converges weakly to F.
\end{lemma}
%
\begin{lemma}[(Theorem~2.1 in \cite{SilversteinC95A})]\label{lesc}
Let $G$ be a function of bounded variation and $x_0\in\mathbb{R}$.
Suppose that $\lim_{z\in\mathbb{C}^+\to x_0}\Im s_G(z)$ exists. Its
limit is
denoted by $\Im s_G(x_0)$. Then $G$ is differentiable at $x_0$, and its
derivative is $\uppi^{-1}\Im s_G(x_0)$.
\end{lemma}

Theorem~\ref{lsd} follows from the following Theorem~\ref{lsdt}.
%
\begin{theorem}\label{lsdt}
Under the conditions \textup{(i)} and \textup{(ii)} in Theorem~\ref{lsd}, we assume that:
\begin{enumerate}[(ii)]
\item[(1)]$\{\A_p\}$ is a sequence of $p\times p$ Hermitian matrices
with uniformly bounded spectral norm in $n$ with probability one
and the ESDFs of $\{\A_p\}$ almost surely tend to a non-random limit
$F^{\A}$ as $p\to\infty$.
\item[(2)] The smallest eigenvalue of matrices $\{\S_n+\alpha_n\A
_p\}$ almost surely tends to a positive value as $n\to\infty$ and
$p\to\infty$.
\end{enumerate}
Then we have $F^{\underline\B_n}\stackrel{a.s.}{\longrightarrow}
\underline F$, where $\underline\B_n=\S_n(\S_n+\alpha_n\A
_p)^{-1}$ and $\underline F$ is a non-random distribution function
whose Stieltjes transform $\underline s=\underline s(z)=s_{\underline
F}(z)$ satisfies
%
\begin{eqnarray}
\label{ssz} \underline s=\int\frac{(1-y(1-z)(z \underline s+1))+\alpha
t}{(1-z)(1-y(1-z)(z \underline s+1))-{ \alpha z}{t}}\,\mathrm{d}F^{\A}(t),
\end{eqnarray}
%
and in the set $\{\underline s\dvtx  \underline s\in\mathbb{C}^{+}\}$ the
solution to \eqref{ssz} is unique.
\end{theorem}
By Lemma~\ref{lebaiz}, we know that to prove Theorem~\ref{lsdt} we
just need to prove three conclusions: (1) $\{F^{\underline\B_n}\}$ is
tight a.s. (2) $s_{F^{\underline\B_n}}\stackrel
{\mathrm{a.s.}}{\longrightarrow} s$ with $s$ satisfying \eqref{ssz}. (3) The
solution to \eqref{ssz} is unique in $\mathbb{C}^{+}$. Now we prove
Theorem~\ref{lsdt} step by step.

\subsection{Proof of Theorem \texorpdfstring{\protect\ref{lsdt}}{2.3}}

\textit{Step 1}: Applying Lemma~\ref{leAB} directly, we have for
any $x_1,x_2\geq0$
%
\begin{eqnarray}\label{eqtt}
F^{\underline\B_n}\bigl\{(x_1x_2,\infty)\bigr\}&\leq&
F^{\S_n}\bigl\{(x_1,\infty )\bigr\}+F^{ (\S_n+\alpha_n\A_p )^{-1}}\bigl
\{(x_2,\infty)\bigr\}
\nonumber
\\[-8pt]\\[-8pt]
&=& F^{\S_n}\bigl\{(x_1,\infty)\bigr\}+F^{(\S_n+\alpha_n\A_p)}
\bigl\{(0,1/x_2)\bigr\} .\nonumber
\end{eqnarray}
It is known that, under the assumptions of Theorem~\ref{lsdt}, with
probability one $F^{\S_n}$ tends to the M--P distribution $F_{\mathrm{mp}}^y$,
which has a density function
%
\begin{eqnarray}
\label{mplaw} f_{\mathrm{mp}}^y(x)= %
\cases{
\displaystyle \frac{1}{2\uppi xy }\sqrt{(b-x) (x-a)},&\quad  \mbox{if }$a\leq x\leq b$,
\cr
0&\quad
\mbox{otherwise,} } %
\end{eqnarray}
and has a point mass $1-1/y$ at the origin if $y>1$,
where $a=(1-\sqrt{y})^2$ and $b=(1+\sqrt{y})^2$.
Thus $\{F^{\S_n}\}$ is tight almost surely, that is, the first term on
the right-hand side of (\ref{eqtt}) can be arbitrarily small by
choosing $x_{1}$ large.

On the other hand, by the second assumption of Theorem~\ref{lsd}, the
second term on the right-hand side of (\ref{eqtt}) can be arbitrarily
small as $n$ is large, provided that $1/x_2$ is smaller than the
smallest eigenvalue of the matrices $\{\S_n+\alpha_n\A_p\}$.
Thus $\{F^{\underline\B_n}\}$ is tight almost surely.

\textit{Step 2}:
Recalling the definition of Stieltjes transform we have that for $z\in
\bbC^{+}$
%
\begin{eqnarray}
s_{F^{\underline\B_n}}(z)=\frac{1}{p}\sum_{i=1}^p
\frac{1}{\la
_i^{\underline\B_n}-z} 
=\frac{1}{p}\tr (
\underline\B_n-z\I )^{-1}.
\end{eqnarray}
Here we have used the fact that $\underline\B_n$ has the same
eigenvalues as
\[
\S_n^{1/2} (\S_n+\alpha_n
\A_p )^{-1}\S_n^{1/2}.
\]
Denote $\underline\B_\ep=\S_n (\S_n+\alpha_n\A_p+\ep\I
)^{-1}$ with small $\ep>0$.
From Lemma~\ref{A.47}, we have
\begin{eqnarray*}
L^3\bigl(F^{\underline\B_n} , F^{\underline\B_\ep}\bigr)\leq
\frac
{1}{n}\tr (\underline\B_n-\underline\B_\ep) (
\underline\B _n-\underline\B_\ep)^*.
\end{eqnarray*}
By the fact
\begin{eqnarray*}
\underline\B_n-\underline\B_\ep&=&\ep
\S_n^{1/2} (\S _n+\alpha_n
\A_p )^{-1/2} (\S_n+\alpha_n
\A_p+\ep\I )^{-1} (\S_n+\alpha_n
\A_p )^{-1/2}\S_n^{1/2}
\\
&\le&\ep (\S_n+\alpha_n\A_p+\ep\I
)^{-1}
\end{eqnarray*}
together with condition
(2) in Theorem~\ref{lsdt},
we obtain almost surely that
$
L^3(F^{\underline\B_n},F^{\underline\B_\ep})\leq
C\ep^2$,
which implies
$
\lim_{\ep\to0}\lim_{n\to\infty}L(F^{\underline\B
_n},F^{\underline\B_\ep})=0$.

Next, we consider the LSDF of $\underline\B_\ep$. Noticing that the
matrix $\alpha_n\A_p+\ep\I$ is invertible for any $\ep>0$, we have
\begin{eqnarray*}
\underline\B_\ep=\I- (\widehat{\underline\B}_\ep+\I
)^{-1},
\end{eqnarray*}
where
$\widehat{\underline\B}_\ep=\S_n(\alpha_n\A_p+\ep\I)^{-1}$.
Thus, we get that
$
F^{\underline\B_\ep}(x)=
F^{\widehat{\underline\B}_\ep}(\frac{1}{1-{x}{}}-1)
$
and
%
\begin{eqnarray}\label{sfbep}
s_{F^{\underline\B_\ep}}(z)= 
\frac{1}{1-z}+
\frac{1}{(1-z)^2}s_{F^{\hat{\underline\B}_\ep
}}\biggl(\frac{z}{1-z}\biggr).
\end{eqnarray}
Silverstein in \cite{Silverstein95S} derived that for any $z\in
\mathbb{C}^+$, the Stieltjes transform of the ESDF of $\widehat
{\underline\B}_\ep$ has a non-random limit, denoted by $s_{\hat\ep
}(z)$, which satisfies the equation
\begin{eqnarray*}
s_{\hat{\ep}}(z)=\int\frac{1}{t(1-y-yzs_{\hat{\ep}}(z))-z}\,\mathrm{d}F^{\A
}_\ep(t),
\end{eqnarray*}
where $F^{\A}_\ep$ is the LSDF of $(\alpha_n\A_p+\ep\I)^{-1}$.
Note that $\Im(z/(1-z))=|1-z|^{-2}\Im z>0$.
Thus by \eqref{sfbep} we get that almost surely $s_{F^{\underline\B
_\ep}}(z)$ tends to a non-random limit, denoted by $\underline s_{\ep
}(z)$, which satisfies
\begin{eqnarray*}
&&(1-z)^2\underline s_\ep(z)-(1-z)\\
&&\quad =\int\frac{1}{t (1-y-y(\afrac
{z}{1-z}) ((1-z)^2\underline s_\ep(z)-(1-z) ) )-\afrac
{z}{1-z}}
\,\mathrm{d}F^{\A}_\ep(t). 
\end{eqnarray*}
%
By definition of $F^{\A}_\ep$ and $F^{\A}$, we have that
\begin{eqnarray*}
\mathrm{d}F^{\A}_\ep(t)=-\mathrm{d}F^{\A}
\biggl(\frac{t^{-1}-\ep}{\alpha}\biggr).
\end{eqnarray*}
%
Therefore letting $\ep\to0$, we have
%
\begin{eqnarray}\label{res1}
\underline s 
 =\int\frac{(1-y(1-z)(z \underline s+1))+\alpha t}{(1-z)(1-y(1-z)(z
\underline s+1))-{ \alpha z}{t}}\,
\mathrm{d}F^{\A}(t).
\end{eqnarray}
%

\textit{Step 3}: From Lemma~\ref{lebaiz}, we conclude that there
exists a distribution function $G$ with support $\Psi_G\subset[0,1]$
satisfying for any $z\in\mathbb C^{+}$,
%
\begin{eqnarray}\label{lasteq}
\underline s(z)=\int_{\Psi_G}\frac{1}{x-z}\,
\mathrm{d}G(x).
\end{eqnarray}
Noticing that $\Im{z}({\alpha+z})^{-1}=\alpha|\alpha+z|^{-2}\Im
z>0$, we infer from \eqref{lasteq} that
\begin{eqnarray*}
\frac{\alpha}{(\alpha+z)^2}\underline s \biggl(\frac{z}{\alpha
+z} \biggr)-
\frac{1}{
\alpha+z}&=&\frac{\alpha}{(\alpha+z)^2}\int_{\Psi_G}
\frac
{1}{x-\afrac{z}{\alpha+z}}\,\mathrm{d}G(x)-\frac{1}{\alpha+z}
\\
&=&\int_{\Psi_G}\frac{1-x}{\alpha x-z(1-x)}\,\mathrm{d}G(x)=\int
_{0}^{\infty
}\frac{1}{x-z}\,\mathrm{d}G \biggl(
\frac{x}{\alpha+x} \biggr).
\end{eqnarray*}
Thus
%
\begin{eqnarray}
\label{us1} \dot\us=\dot\us(z)=\frac{\alpha}{(\alpha+z)^2}\underline s \biggl(
\frac{z}{\alpha+z} \biggr)-\frac{1}{
\alpha+z}
\end{eqnarray}
is a Stieltjes transform of the distribution function $G(\frac
{x}{\alpha+x})$ with $x\in[0,\infty)$.
Notice that even if $G(x)$ has a point mass at $x=1$, we have $\frac
{1-x}{\alpha x-z(1-x)}=0$.
Thus, \eqref{res1} can be represented as
\begin{eqnarray*}
\dot\us(z)=\int_{\bbR^+}\frac{1}{t(1-y-yz \dot \us
(z))-z}\,\mathrm{d}
\biggl(1-F^{\A}\biggl(\frac{1}{ t}\biggr)\biggr),
\end{eqnarray*}
where $\bbR^+=\{t\dvtx t\in\bbR,t>0\}$.
It is shown that the solution of the last equation is unique in
$\mathbb C^+$ (see \cite{Silverstein95S}). Thus,
we obtain that \eqref{res1} has a unique solution in $\mathbb{C}^+$,
which completes the proof of Theorem~\ref{lsdt}.

\subsection{Proof of Theorem \texorpdfstring{\protect\ref{lsd}}{1.1}}
Using Theorem~\ref{lsdt} and Remark~\ref{remark1.2}, we know that the
Stieltjes transform of $F^{}$ is the unique solution in $\mathbb{C}^+$
to the equation
%
\begin{eqnarray}\label{sti1}
s=\int\frac{(1-y(1-z)(z s+1))+\alpha t}{(1-z)(1-y(1-z)(z s+1))-{
\alpha z}{t}}\,\mathrm{d}F^{Y}_{\mathrm{mp}}(t).
\end{eqnarray}
Here $F^{Y}_{\mathrm{mp}}$ is the limit of $F^{\T_N}$ which is also the M--P
distribution.
After elementary calculations, we may represent the last equation as
%
\begin{eqnarray}\label{sss}
s=-\frac{1}{z}-\frac{\vp}{\alpha z^2}\int\frac{1}{t-\sklafrac
{(1-z)\vp}{ \alpha z}}\,
\mathrm{d}F^Y_{\mathrm{mp}}(t),
\end{eqnarray}
where $\vp=1-y(1-z)(z s+1)$.
Recalling \eqref{us1}, we have that
\begin{eqnarray*}
s(z)=\frac{1}{1-z}+\frac{\alpha}{(1-z)^2}\dot s\biggl(\frac{\alpha z}{1-z}\biggr)
\end{eqnarray*}
and
\begin{eqnarray*}
\vp=1-y(1-z) (z s+1)=1- y-\frac{\alpha yz}{(1-z)}\dot s\biggl(\frac{\alpha z}{1-z}
\biggr),
\end{eqnarray*}
which implies
\begin{eqnarray*}
\frac{(1-z)\vp}{ \alpha z}= \frac{(1-z)(1- y)}{ \alpha z}-y\dot s\biggl(\frac{\alpha z}{1-z}\biggr).
\end{eqnarray*}
Noticing that $\Im\frac{\alpha z}{1-z}>0$, we have $\Im \frac
{(1-z)\vp}{ \alpha z}<0$ and
\begin{eqnarray*}
\int\frac{1}{t-\sklafrac{(1-z)\vp}{ \alpha z}}\,\mathrm{d}F^Y_{\mathrm{mp}}(t)=\overline
{s^Y_{\mathrm{mp}} \biggl(\frac{(1-\bar z)\vp(\bar z)}{ \alpha\bar z} \biggr)},
\end{eqnarray*}
where $s^Y_{\mathrm{mp}}$ is the Stieltjes transform of the M--P distribution
$F^Y_{\mathrm{mp}}$. Since
\begin{eqnarray*}
s^Y_{\mathrm{mp}}(z)=\frac{1-Y-z+\sqrt{(z-1-Y)^2-4Y}}{2Yz},
\end{eqnarray*}
the equation \eqref{sss} implies
\begin{eqnarray*}
s=-\frac{1}{z}-\frac{\vp}{\alpha z^2} \biggl(\frac{1-Y-\sklafrac
{(1-z)\vp}{ \alpha z}+\sqrt{(\sklafrac{(1-z)\vp}{ \alpha
z}-1-Y)^2-4Y}}{2Y\sklafrac{(1-z)\vp}{ \alpha z}} \biggr),
\end{eqnarray*}
where, and throughout this section, the square-root of a complex number
is specified as the one with positive imaginary part.
The solution to this equation is
\begin{eqnarray*}
s(z)= \frac{(1+y)(1-z)-\alpha z(1-Y)+\sqrt{((1-y)(1-z)+\alpha
z(1-Y))^2-4\alpha z(1-z)}}{2 z(1-z)(y(1-z)+\alpha zY)}- \frac{1}{ z}.
\end{eqnarray*}
Now using Lemma~\ref{lesc} and letting $z\downarrow x+\mathrm{i}0$, $\uppi
^{-1}\Im s(z)$ tends to the density function of the LSDF of $\B_n$.
Thus, the density function of the LSDF of $\B_n$ is
\begin{eqnarray*}
\lleft\{ %
\begin{array} {l} \displaystyle \frac{\sqrt{4\alpha x(1-x)-((1-y)(1-x)+\alpha x(1-Y))^2}}{2 \uppi
x(1-x)(y(1-x)+\alpha xY)}, \\
\hphantom{0,\qquad\mbox{}} \mbox{if $4\alpha
x(1-x)-\bigl((1-y) (1-x)+\alpha x(1-Y)\bigr)^2>0$;}
\\
0,\qquad \mbox{otherwise.} \end{array} %
\rright.
\end{eqnarray*}
Or equivalently,
\begin{eqnarray*}
\lleft\{ %
\begin{array} {l@{\qquad}l} \displaystyle \frac{\sqrt{((\alpha(1-Y)-1+y)^2+4\alpha)(x_r-x)(x-x_l)}}{2 \uppi
x(1-x)(y(1-x)+\alpha xY)}, & \mbox{if
$x_l<x<x_r$;}
\\
0, & \mbox{otherwise,} \end{array} %
\rright.
\end{eqnarray*}
where $x_l,x_r= (\frac{2\alpha-(1-y)[\alpha(1-Y)-1+y]\mp
2\alpha\sqrt{y-yY+Y}}{(\alpha(1-Y)-1+y)^2+4\alpha} )$. Now we
determine the possible atom at $0$ and $1$. When $z\to0$ with $\Im
z>0$, we have
\begin{eqnarray*}
&&\Im\bigl[\bigl((1-y) (1-z)+\alpha z(1-Y)\bigr)^2-4\alpha z(1-z)
\bigr]
\\
&&\quad =2\Im z \bigl\{ \bigl[(1-Y)\alpha-1+y \bigr] \bigl[(1-y) (1-\Re z)+\alpha (1-Y)
\Re z \bigr]-2\alpha(1-2\Re z) \bigr\}<0.
\end{eqnarray*}
By the fact that the real part of $\sqrt{g(z)}$ has the same sign as
that of the imaginary part of $g(z)$, we obtain that $\Re\sqrt {((1-y)(1-z)+\alpha z(1-Y))^2-4\alpha z(1-z)}<0$. Thus
\begin{eqnarray*}
\sqrt{\bigl((1-y) (1-z)+\alpha z(1-Y)\bigr)^2-4\alpha z(1-z)}
\to-|1-y|.
\end{eqnarray*}
Consequently,
\begin{eqnarray*}
F\{0\}=-\lim_{z\to0}zs(z)=\frac{|1-y|-1-y}{2y}+1 =\lleft\{
\begin{array} {l@{\qquad}l} \displaystyle \frac{y-1}{y}, & \mbox{if $y>1$;}
\\
0, & \mbox{otherwise.} \end{array} %
\rright.
\end{eqnarray*}

When $z\to1$ with $\Im z>0$, we have
\begin{eqnarray*}
&&\Im\bigl[\bigl((1-y) (1-z)+\alpha z(1-Y)\bigr)^2-4\alpha z(1-z)
\bigr]
\\
&&\quad =2\Im z \bigl\{ \bigl[(1-Y)\alpha-1+y \bigr] \bigl[(1-y) (1-\Re z)+\alpha (1-Y)
\Re z \bigr]-2\alpha(1-2\Re z) \bigr\}>0.
\end{eqnarray*}
Hence, we get $\Re\sqrt{((1-y)(1-z)+\alpha z(1-Y))^2-4\alpha
z(1-z)}>0$. Thus,
\begin{eqnarray*}
\sqrt{\bigl((1-y) (1-z)+\alpha z(1-Y)\bigr)^2-4\alpha z(1-z)}\to
\alpha|1-Y|.
\end{eqnarray*}
Consequently,
\begin{eqnarray*}
F\{1\}=-\lim_{z\to1}(z-1)s(z)=\frac{|1-Y|-(1-Y)}{2 Y} =\lleft\{
\begin{array} {l@{\quad}l} \displaystyle \frac{Y-1}{Y}, &\quad  \mbox{if $Y>1$;}
\\
0, &\quad  \mbox{otherwise.} \end{array} %
\rright.
\end{eqnarray*}
Then the proof of Theorem~\ref{lsd} is complete.

\section{Framework of proving Theorem \texorpdfstring{\protect\ref{clt}}{1.6}}\label{sec3}
In this section, we will give the proof of Theorem~\ref{clt}. Recall
the definition of the Stieltjes transform of a distribution function $G(x)$.
Now we extend the Stieltjes transform to the whole complex plane except
the interval $[c_l,c_r]$ analytically.
Since every $f_k(x)$ is analytic on an open region containing the
interval $[c_l,c_r]$, we assume that the analytic region contains the
contour $\cC=\{z\in\mathbb{C}\dvtx \Re z\in[c_l-\th,c_r+\th],\Im z=\pm
\th\}\cup\{z\in\mathbb{C}\dvtx \Re z\in\{c_l-\th,c_r+\th\},\Im
z\in[- \th,\th]\}$. Here $\th$ can be small enough. By Cauchy's
integral formula
\begin{eqnarray*}
f_k(x)=\frac{1}{2\uppi \mathrm{i}}\oint_\cC\frac{f_k(z)}{z-x}\,
\mathrm{d}z,
\end{eqnarray*}
we have for $l\geq1$ and complex constants $a_1,\ldots,a_l$,
%
\begin{eqnarray}\label{cauchy}
\sum_{k=1}^l a_kp \biggl(
\int f_k(x)\,\mathrm{d}F^{\B_n}(x)-\int f_k(x)\,
\mathrm{d}F_{0}(x) \biggr) =-\sum_{k=1}^l
\frac{a_k}{2\uppi  \mathrm{i}}\oint _{\mathcal{C}} f_k(z) S_n(z)\,
\mathrm{d}z,
\end{eqnarray}
where $S_n(z)=p(s_n(z)-s_0(z))$ and $s_0(z)$ is the Stieltjes transform
of $F$ with constants $y$ and $Y$ replaced by
$y_n=p/n$ and $Y_n=p/N$. We remind the readers to notice that the above
equality may not be correct when some eigenvalues of $\B_n$ fall
outside the contour. However, by Remark~\ref{remark1.2}, Lemma~\ref
{le5.7} and the exact separation theorem in \cite{BaiS98N}, we know
for $y>1\ (\mbox{or }Y>1)$
and sufficiently large $n\ (\mbox{or }N)$, the mass at the origin (one)
of $F^{\B_n}$ will coincide exactly with that of $F_0$ and with
overwhelming probability all the other eigenvalues of $\B_n$ fall in
$[c_l-\th,c_r+\th]$. Thus to prove Theorem~\ref{clt}, it suffices
for us to derive the limiting distribution of \eqref{cauchy}.

Write
\begin{eqnarray*}
S_n(z)=p\bigl(s_n(z)-s_{N0}(z)\bigr)+p
\bigl(s_{N0}(z)-s_0(z)\bigr):=S_{n1}+S_{n2},
\end{eqnarray*}
where $s_{N0}(z)$ is the unique root of the equation
\begin{eqnarray*}
s_{N0}=\int\frac{(1-y_n(1-z)(z s_{N0}(z)+1))+\alpha_n
t}{(1-z)(1-y_n(1-z)(z s_{N0}+1))-{ \alpha_n z}{t}}\,\mathrm{d}F^{\T_N}(t)
\end{eqnarray*}
in the set $\{s_{N0}(z)\in\bbC^+\}$. Using the notation\vspace*{2pt}
$\dot s_{N0}=\dot s_{N0}(z)=\frac{\alpha_n}{(\alpha_n+z)^2}
s_{N0}(\frac{z}{\alpha_n+z})-\frac{1}{\alpha_n+z}$, $\dot
s_{0}=\dot s_{0}(z)=\frac{\alpha_n}{(\alpha_n+z)^2} s_{0}(\frac
{z}{\alpha_n+z})-\frac{1}{\alpha_n+z}$, $\ddot
s_{N0}(z)=-z^{-1}(1-y_n)+y_n\dot s_{N0}(z)$ and\vspace{2pt} $\ddot
s_{0}(z)=-z^{-1}(1-y_n)+y_n\dot s_{0}(z)$ we have
\begin{eqnarray*}
z=-\frac{1}{\ddot s_{N0}}+y_n\int\frac{\mathrm{d}F^{\T_N}(t)}{ t+\ddot
s_{N0}}\quad \mbox{and}\quad  z=-
\frac{1}{\ddot s_{0}}+y_n\int\frac{\mathrm{d}F_{\mathrm{mp}}^{Y_N}(t)}{ t+\ddot s_{0}}.
\end{eqnarray*}
Making difference of the two identities above yields that
\begin{eqnarray*}
\frac{\ddot s_{0}-\ddot s_{N0}}{\ddot s_{0}\ddot s_{N0}}=y_n\int\frac
{(\ddot s_{0}-\ddot s_{N0})\,\mathrm{d}F^{\T_N}(t)}{( t+\ddot s_{N0})( t+\ddot
s_{0})}+y_n\int
\frac{\mathrm{d}F^{\T_N}(t)-\mathrm{d}F_{\mathrm{mp}}^{Y_N}(t)}{ t+\ddot s_{0}}.
\end{eqnarray*}
Then we get
%
\begin{eqnarray}\label{phil1}
{\ddot s_{0}-\ddot s_{N0}} {}=y_n\ddot
s_{0}\ddot s_{N0}\int\frac
{\mathrm{d}F^{\T_N}(t)-\mathrm{d}F_{\mathrm{mp}}^{Y_N}(t)}{ t+\ddot s_{0}}
\biggl(1-y_n\ddot s_{0}\ddot s_{N0}\int
\frac{\mathrm{d}F^{\T_N}(t)}{( t+\ddot s_{N0})( t+\ddot
s_{0})}\biggr)^{-1}.
\end{eqnarray}
Let $s^{\T_N}$ be the Stieltjes transforms of $F^{\T_N}$ and then
from (6.32) in \cite{Zheng12C} we have the conclusion that
\[
p\int\frac{\mathrm{d}F^{\T_N}(t)-\mathrm{d}F_{\mathrm{mp}}^{Y_N}(t)}{ t+\ddot s_{0}}=p\bigl(s^{\T
_N}(-\ddot s_{0})-s^{Y_N}_{\mathrm{mp}}(-
\ddot s_{0})\bigr)
\]
converges weakly to a Gaussian process $\Phi_1$ on $\cC$ with mean function
%
\begin{eqnarray}\label{phi1m}
\E\Phi_1(z)=\frt\frac{Y[\dddot s(z)]^3[1+\dddot s(z)]^{-3}}{\{
1-Y\dddot s(z)/[1+\dddot s(z)]^2\}^2} +(\frm_\mathbbl{x}-
\frt-2)\frac{Y[\dddot s(z)]^3[1+\dddot
s(z)]^{-3}}{1-Y[\dddot s(z)]^2/[1+\dddot s(z)]^{2}}
\end{eqnarray}
and covariance function
%
\begin{eqnarray}\label{phi1c}
\cov\bigl(\Phi_1(z_1),\Phi_1(z_2)
\bigr)&=&(\frt+1) \biggl(\frac{(\dddot
s(z_1))'(\dddot s(z_2))'}{[\dddot s(z_1)-\dddot s(z_2)]^2}-\frac
{1}{(\ddot s(z_1)-\ddot s(z_2))^2} \biggr)
\nonumber
\\[-8pt]\\[-8pt]
&&{}+(\frm_\mathbbl{x}-\frt-2)\frac{Y(\dddot s(z_1))'(\dddot
s(z_2))'}{[1+\dddot s(z_1)]^{2}[1+\dddot s(z_2)]^{2}},\nonumber
\end{eqnarray}
where $\dddot s(z)=\ddot s_{\mathrm{mp}}^Y(-\ddot s(z))$, $\ddot
s_{\mathrm{mp}}^Y(z)=-z^{-1}(1-Y)+Ys_{\mathrm{mp}}^Y(z)$ and $(\dddot s(z_i))'=\frac
{\mathrm{d}}{\mathrm{d}z}\ddot s_{\mathrm{mp}}^Y(z)|_{z=-\ddot s(z_i)}$, $i=1,2$.
And $\{S_{n2}(\cdot)\}$ forms
a tight sequence on $\cC$ and $S_{n2}(\frac{z}{\alpha+z})$ converges
weakly to a Gaussian process $-(\alpha+z)^2\ddot s'(z)\Phi_1(z)$ with
mean function
\begin{eqnarray*}
\E\bigl(-(1+z)^2\ddot s'(z)\Phi_1(z)
\bigr)=-(\alpha+z)^2\ddot s'(z)\cdot \eqref{phi1m}
\end{eqnarray*}
and covariance function
\begin{eqnarray*}
&&\cov\bigl(-(\alpha+z_1)^2\ddot s'(z_1)
\Phi_1(z_1),-(\alpha+z_2)^2\ddot
s'(z_2)\Phi_1(z_2)\bigr)
\\
&&\quad =(\alpha+z_1)^2(\alpha+z_2)^2
\ddot s'(z_1)\ddot s'(z_2)
\cdot\eqref{phi1c}.
\end{eqnarray*}

Recall the notation $\vp=1-y(1-z)(z s+1)$ and
suppose we have the following lemma.
%
\begin{lemma}\label{leclt}
Under the conditions of Theorem~\ref{lsd} and $z\in\cC$, we have
that given $\mathfrak{T}_N=\{\mbox{all } \T_N\}$, $\{S_{n1}(\cdot)\}
$ forms
a tight sequence on $\cC$ and $S_{n1}(z)$ converges weakly to a
two-dimensional Gaussian process $\Phi_2(z)$ satisfying
%
\begin{eqnarray}\label{phi3}
&&\E\bigl(\Phi_2(z)|\mathfrak{T}_N\bigr)\nonumber\\
&&\quad = \frt
\frac{\int\sklfrac{\alpha y
(1-z)\vp^3 t}{ ((1-z)\vp-z\alpha t
)^3}\,\mathrm{d}F_{\mathrm{mp}}^{Y}(t)}{ (1-y\int\sklfrac{(1-z)^2\vp^2 }{
((1-z)\vp-z\alpha t )^2}\,\mathrm{d}F_{\mathrm{mp}}^{Y}(t) )^2}
\nonumber\\[-8pt]\\[-8pt]
&&\qquad {}+(\frm_x-\frt-2) (1-z)y\vp^3\nonumber \\
&&\quad \qquad \,{}{}\times\frac{\int\sklfrac{1}{(1-z)\vp-z\alpha
t}\,\mathrm{d}F_{\mathrm{mp}}^{Y}(t)\int\sklfrac{\alpha t}{ ((1-z)\vp-z\alpha t
)^2}\,\mathrm{d}F_{\mathrm{mp}}^{Y}(t)}{1-y\int\sklfrac{(1-z)^2\vp^2 }{ ((1-z)\vp
-z\alpha t )^2}\,\mathrm{d}F_{\mathrm{mp}}^{Y}(t)}\nonumber
\end{eqnarray}
and
%
\begin{eqnarray}\label{phi4}
&&\cov\bigl(\Phi_2(z_1),\Phi_2(z_2)|
\mathfrak{T}_N\bigr)
\\
&&\quad =\frac{\partial^2}{\partial z_1\,\partial z_2} \biggl((\frt+1)\nonumber\\
&&\hphantom{\quad =\frac{\partial^2}{\partial z_1\,\partial z_2} \biggl(}{}\times\int
\biggl[
\biggl(\int\frac{y(1-z_1)(1-z_2)\vp(z_1)\vp(z_2)}{ ((1-z_1)\vp
-z_1\alpha t ) ((1-z_2)\vp-z_2\alpha t
)}\,\mathrm{d}F_{\mathrm{mp}}^{Y}(t)
\biggr)\nonumber\\
&&\hphantom{\quad =\frac{\partial^2}{\partial z_1\,\partial z_2} \biggl(\times\int
\biggl[}{}\times \biggl(1-t\int\frac{y(1-z_1)(1-z_2)\vp(z_1)\vp
(z_2)}{ ((1-z_1)\vp-z_1\alpha t ) ((1-z_2)\vp
-z_2\alpha t )}\,\mathrm{d}F_{\mathrm{mp}}^{Y}(t)\biggr)^{-1}\,\mathrm{d}t\biggr]
\nonumber\\
&&\hphantom{\quad =\frac{\partial^2}{\partial z_1\,\partial z_2} \biggl(}{}+(\frm_x-\frt-2)y\nonumber\\[-8pt]\\[-8pt]
&&\hphantom{\quad =\frac{\partial^2}{\partial z_1\,\partial z_2} \biggl( + (}{}\times\int\frac{(1-z_1)\vp(z_1)}{(1-z_1)\vp
-z_1\alpha t}\,
\mathrm{d}F_{\mathrm{mp}}^{Y}(t)\int\frac{(1-z_2)\vp(z_2)}{(1-z_2)\vp
-z_2\alpha t}\,
\mathrm{d}F_{\mathrm{mp}}^{Y}(t) \biggr).\nonumber
\end{eqnarray}

\end{lemma}
We postpone the proof of this lemma to the next section. Now we use the
notation $\dot s=\dot s(z)=\frac{\alpha}{(\alpha+z)^2} s(\frac
{z}{\alpha+z})-\frac{1}{\alpha+z}$ and $\ddot
s(z)=-z^{-1}(1-y)+y\dot s(z)$ to get
%
\begin{eqnarray}\label{vpddot}
\vp(z)=-\frac{\alpha z}{1-z}\ddot s\biggl(\frac{\alpha z}{1-z}\biggr),
\end{eqnarray}
which can be used to rewrite \eqref{phi3} and \eqref{phi4} as
%
\begin{eqnarray}\label{phi21}
&&\E\biggl(\Phi_2\biggl(\frac{z}{\alpha+z}\biggr)\Bigl|
\mathfrak{T}_N\biggr)\nonumber\\[-8pt]\\[-8pt]
&&\quad \to\frt\frac
{y(\alpha+z)^2\int\alpha t (\ddot s(z) )^3 (\ddot
s(z)+ t )^{-3}\,\mathrm{d}F_{\mathrm{mp}}^{Y}(t)}{ (1-y\int (\ddot
s(z) )^2 (\ddot s(z)+ t )^{-2}\,\mathrm{d}F_{\mathrm{mp}}^{Y}(t)
)^2}
\nonumber\\
\label{phi5}&&\qquad {}+(\frm_x-\frt-2)\nonumber\\[-8pt]\\[-8pt]
&&\quad \qquad {}\times \frac{y(\alpha+z)^2\int\sklafrac{ \ddot s(z)}{\ddot
s(z)+ t}\,\mathrm{d}F_{\mathrm{mp}}^{Y}(t)\int\sfrac{\alpha t (\ddot s(z)
)^2}{ (\ddot s(z)+ t )^{2}}\,\mathrm{d}F_{\mathrm{mp}}^{Y}(t)}{1-y\int
(\ddot s(z) )^2 (\ddot s(z)+ t
)^{-2}\,\mathrm{d}F_{\mathrm{mp}}^{Y}(t)}\nonumber
\end{eqnarray}
and
%
\begin{eqnarray}\label{phi6}
&&\cov\biggl(\Phi_2\biggl(\frac{z_1}{\alpha+z_1}\biggr),
\Phi_2\biggl(\frac{z_2}{\alpha
+z_2}\biggr)\Bigl|\mathfrak{T}_N
\biggr)
\nonumber
\\
&&\quad \to(\frt+1) (\alpha+z_1)^2(\alpha+z_2)^2
\biggl(\frac
{\ddot s'(z_1)\ddot s'(z_2)}{(\ddot s(z_1)-\ddot s(z_2))^2}-\frac
{1}{(z_1-z_2)^2} \biggr)\nonumber
\\[-8pt]\\[-8pt]
&&\qquad {}+(\frm_x-\frt-2)y(\alpha+z_1)^2(
\alpha+z_2)^2\nonumber\\
&&\qquad \quad \,{}{}\times\int\frac{\alpha
t\ddot s'(z_1)}{ (\ddot s(z_1)+ t )^{2}}\,
\mathrm{d}F_{\mathrm{mp}}^{Y}(t)\int \frac{\alpha t\ddot s'(z_2)}{ (\ddot s(z)+ t
)^{2}}\,
\mathrm{d}F_{\mathrm{mp}}^{Y}(t).\nonumber
\end{eqnarray}
Here we used the fact that (similar to \eqref{phil1})\vspace*{-1pt}
\begin{eqnarray*}
z_1-z_2=\frac{\ddot s(z_1)-\ddot s(z_2)}{\ddot s(z_1)\ddot
s(z_2)} \biggl(1-y\int\ddot
s(z_1)\ddot s(z_2) \bigl(\ddot s(z_1)+ t
\bigr)^{-1} \bigl(\ddot s(z_2)+ t \bigr)^{-1}\,
\mathrm{d}F_{\mathrm{mp}}^{Y}(t) \biggr).
\end{eqnarray*}

As the mean and covariance of the limiting distribution are independent
of the conditioning $\mathfrak{T}_N$, we conclude that $S_{n1}$ and
$S_{n2}$ are asymptotically independent. Then from the above argument
and page 473 in \cite{Zheng12C} we can get that $S_n(\frac{z}{1+z})$
converges weakly to a Gaussian process $-(1+z)^2\ddot s'(z)\Phi
_1(z)+\Phi_2(\frac{z}{1+z})$ and
together with \eqref{cauchy} and Lemma~\ref{limit} implies Theorem~\ref{clt}.

\section{Proof of Lemma \texorpdfstring{\protect\ref{leclt}}{3.1}}\label{sec4}
In this section, we give the proof of Lemma~\ref{leclt}. Following the
similar truncation steps in \cite{BaiS04C} we may truncate and
renormalize the
random variables $\{x_{ij}\}$ as follows:\vspace*{-1.5pt}
\begin{eqnarray*}
|x_{ij}|\leq\delta_n\sqrt{n}, \qquad \E x_{ij}=0 \quad \mbox{and}\quad  \E|x_{ij}|^2=1.
\end{eqnarray*}
Here $\delta_n\to0$ which can be arbitrarily slow. Based on this
truncation, we can verify that:\vspace*{-1.5pt}
%
\begin{eqnarray}
\E| x_{ij}|^4=\frm_x+\mathrm{o}(1),
\end{eqnarray}
and if $\X_n$ is complex valued,\vspace*{-1pt}
\begin{eqnarray*}
\E x_{ij}^2=\mathrm{O}\bigl(n^{-1}\bigr).
\end{eqnarray*}
%
We will introduce some notation and provide some bounds in the first
part of this section. The proof of Lemma~\ref{leclt} will be given in
the next part. The main procedures of the proofs, including the
Stieltjes transform, the martingale decomposition and Burkholder's
inequality, are routine in RMT, hence we will outline them without
detailed descriptions. Interested readers are referred to Bai and
Silverstein \cite{BaiS10S}. Throughout the rest of the paper, constants
appearing in inequalities are represented by $C$ which are nonrandom
and may take different
values from one appearance to another.
\subsection{Definitions and some basic results}

In this part,
we introduce some notation and some useful results. First, we assume
$z=u+\mathrm{i}\th$ with $\th>0$. For simplicity, write $\S=\S_n$ and $\B
=\B_n$.
Let $\D=\D(z)=\B-z\I$, $\F=\F(z)=({1-z})\S-{z\alpha_n}\T_N$
and $\I$ be the identity matrix. Define $\r_i=n^{-1/2}\X_{(\cdot
i)}$ where $\X_{(\cdot i)}$ is the $i$th column of $\X_n$, $\S_i=\S
-\r_i\r_i^*$, $\B_i=\S_i(\S_i+\alpha_n\T_N)^{-1}$, $\D_i=\D
_i(z)=\B_i-z\I$ and $\F_i=\F_i(z)=({1-z})\S_i-{z\alpha_n}\T_N$. Let
$\E_i=\E(\cdot|\mathfrak{T}_N,\r_1,\dots,\r_i)$ and $\E_0=\E
(\cdot|\mathfrak{T}_N)$. Moreover, introduce\vspace*{-1pt}
\begin{eqnarray*}
\vp_i&=&\vp_i(z)=\frac{1}{1+({1-z})\r_i^*\F_i^{-1}(z)\r_i}, \qquad \vp
_i^{\mathrm{tr}}=\varpi^{\mathrm{tr}}_i(z)=
\frac{1}{1+n^{-1}({1-z})\tr\F_i^{-1}(z)},
\\[-1pt]
\vp_i^{\E}&=&\varpi^{\E}_i(z)=
\frac{1}{1+n^{-1}({1-z})\E_0 \tr\F
_i^{-1}(z)},
\\
\gamma_i&=&\ga_i(z)=\r_i^*
\F_i^{-1}\r_i-n^{-1}
\E_0 \tr\F_i^{-1},\qquad  \eta_i=
\eta_i(z)=\r_i^*\F_i^{-1}
\r_i-n^{-1} \tr\F_i^{-1},
\\
\xi_i&=&\xi_i(z)=n^{-1} \tr
\F_i^{-1}-n^{-1} \E_0 \tr
\F_i^{-1},
\\
s_n&=&s_n(z)=s_{F^{\mathbf{B}_n}}(z),\qquad  s=s(z)=s_{F^{y,H}}(z),\qquad
s_0=s_0(z)=s_{F^{y_n,H_n}}(z).
\end{eqnarray*}
Obviously we have,
%
\begin{eqnarray}
\label{gai}\ga_i(z)&=&\eta_i(z)+\xi_i(z),
\\
\label{evpi}\vp_i&=&\varpi^{\E}_i-(1-z)
\varpi^{\E}_i\vp_i\ga_i=
\varpi^{\E
}_i-(1-z) \bigl(\varpi^{\E}_i
\bigr)^2\ga_i+(1-z)^2\bigl(
\varpi^{\E}_i\bigr)^2\vp_i\ga
_i^2
\end{eqnarray}
and
%
\begin{eqnarray}
\vp_i=\varpi^{\mathrm{tr}}_i-(1-z)
\varpi^{\mathrm{tr}}_i\vp_i\eta_i=\varpi
^{\mathrm{tr}}_i-(1-z) \bigl(\varpi^{\mathrm{tr}}_i
\bigr)^2\eta_i+(1-z)^2\bigl(
\varpi^{\mathrm{tr}}_i\bigr)^2\vp _i
\eta_i^2\label{trvpi}.
\end{eqnarray}

It is easy to verify that
%
\begin{eqnarray}
\Im(1-z)^{-1}={\th}|1-z|^{-2}\label{im}
\end{eqnarray}
and
%
\begin{eqnarray}
\Im\r_i^*\F_i^{-1}(z)\r_i =\th
\r_i^*\F_i^{-1}(z) (\S_i+
\alpha_n\T_N )^{}\F _i^{-1}(
\bar{z})\r_i\label{im1}
\end{eqnarray}
have the same sign.
Therefore from the definition of $\vp_i$, we have
%
\begin{eqnarray}\label{bovpi}
|\vp_i|=\biggl\llvert \frac{1}{1-z}\frac{1}{\afrac{1}{1-z}+\r_i^*\F
_i^{-1}(z)\r_i}\biggr
\rrvert \leq\frac{|1-z|}{\th}.
\end{eqnarray}
Similarly we can obtain
%
\begin{eqnarray}\label{boetrvpi}
\bigl|\vp_i^{\mathrm{tr}}\bigr|\leq\frac{|1-z|}{\th}, \qquad \bigl|
\vp_i^{\E}\bigr|\leq\frac
{|1-z|}{\th}.
\end{eqnarray}
By the fact that
%
\begin{eqnarray}
\bigl\|\F_i^{-1}(z)\bigr\|=\bigl\|\D_i^{-1}(z) (
\S_i+\alpha_n\T_N)^{-1}\bigr\|\leq C\th
^{-1}\label{fi}
\end{eqnarray}
and Lemma~\ref{leby}, we have for any $l\geq2$
\begin{equation}\label{eetail}
\E\bigl|\eta_i(z)\bigr|^l\leq\frac{C\delta_n^{2l-4}}{n\th^l}.
\end{equation}
In the last inequality we used $|x_{ij}|\leq\delta_n\sqrt{n}$.
For any invertible matrices $\M$, $\M+\r_i\r_i^*$ and $\N$, using
%
\begin{eqnarray}\label{mn}
\r_i^*\bigl(\M+ \r_i\r_i^*
\bigr)^{-1}=\frac{1}{1+ \r_i^*\M\r_i}\r_i^*\M ^{-1},\qquad
\M^{-1}-\N^{-1}=-\N^{-1}(\M-\N)\M^{-1},
\end{eqnarray}
we obtain that\vspace*{-1pt}
%
\begin{eqnarray}\label{ffi}
\F^{-1}(z)-\F_i^{-1}(z)=-(1-z)
\varpi_i\F_i^{-1}\r_{i}
\r_{i}^*\F _i^{-1},
\end{eqnarray}
which together with (\ref{im1})--\eqref{fi} implies that for any
Hermitian matrix $\M$ with $\|\M\|\leq C$,\vspace*{-1pt}
%
\begin{eqnarray}\label{ddi}
\bigl|\tr\F^{-1}(z)\M-\tr\F_i^{-1}(z)\M\bigr|=\bigl|(1-z)
\varpi_i\r_{i}^*\F _i^{-1}\M
\F_i^{-1}\r_{i}\bigr|\leq C\th^{-1}.
\end{eqnarray}
%
%
\begin{lemma}\label{efef}
Under the conditions of Theorem~\ref{clt}, we have for any non-random
Hermitian matrix $\M$ with $\|\M\|\leq C$ and $l\geq2$,\vspace*{-1pt}
\begin{eqnarray*}
\E\bigl|n^{-1}\tr\F^{-1}(z)\M-n^{-1}\E_0
\tr\F^{-1}(z)\M\bigr|^l\leq \frac{C_l\delta_n^{2l-4}}{n^{l/2+1}\th^{3l}}, \qquad \mbox{where }
z=u+\mathrm{i}\theta.
\end{eqnarray*}
\end{lemma}
\begin{pf}
The martingale decomposition (one can refer to \cite{BaiS10S} for more
details) gives\vspace*{-1.5pt}
\begin{eqnarray*}
\tr\F^{-1}\M-\E_0 \tr\F^{-1}\M&=&\sum
_{i=1}^n(\E_i-\E_{i-1})\tr
\bigl(\F^{-1}\M-\F_i^{-1}\M\bigr)
\\[-1pt]
&=&- (1-z) \sum_{i=1}^n(\E_i-
\E_{i-1})\varpi_i\r_{i}^*\F_i^{-1}
\M \F_i^{-1}\r_{i}
\\[-1pt]
&=& (z-1) \sum_{i=1}^n(\E_i-
\E_{i-1})\varpi^{\mathrm{tr}}_i\r_{i}^*\F
_i^{-1}\M\F_i^{-1}
\r_{i}\\[-1pt]
&&{}+(1-z)^2 \sum_{i=1}^n(
\E_i-\E _{i-1})\varpi^{\mathrm{tr}}_i
\vp_i\eta_i \r_{i}^*\F_i^{-1}
\M\F_i^{-1}\r_{i}.
\end{eqnarray*}
Here we used \eqref{ffi} and \eqref{trvpi}. From \eqref{fi} and
Lemma~\ref{leby}, we obtain that\vspace*{-1pt}
\begin{eqnarray*}
\E\bigl|\r_{i}^*\F_i^{-1}\M
\F_i^{-1}\r_{i}-n^{-1}\tr
\F_i^{-1}\M\F _i^{-1}\bigr|^l
\leq\frac{C\delta_n^{2l-4}}{n\th^{2l}}.
\end{eqnarray*}
Thus it follows from \eqref{boetrvpi} and Lemma~\ref{bhi} that\vspace*{-1pt}
\begin{eqnarray*}
&&\E\Biggl\llvert \sum_{i=1}^n(
\E_i-\E_{i-1})\varpi^{\mathrm{tr}}_i
\r_{i}^*\F _i^{-1}\M\F_i^{-1}
\r_{i}\Biggr\rrvert ^l
\\[-1pt]
&&\quad =\E\Biggl\llvert \sum_{i=1}^n(
\E_i-\E_{i-1})\varpi^{\mathrm{tr}}_i \bigl(\r
_{i}^*\F_i^{-1}\M\F_i^{-1}
\r_{i}-n^{-1}\tr\F_i^{-1}\M\F
_i^{-1} \bigr)\Biggr\rrvert ^l \leq
\frac{Cn^{l/2}\delta_n^{2l-4}}{n\th^{3l}}.
\end{eqnarray*}
On the other hand, from \eqref{boetrvpi}, \eqref{ddi}, \eqref
{eetail} and Lemma~\ref{bhi} we also have\vspace*{-1pt}
\begin{eqnarray*}
\E\Biggl\llvert \sum_{i=1}^n(
\E_i-\E_{i-1})\varpi^{\mathrm{tr}}_i
\vp_i\eta_i \r _{i}^*\F_i^{-1}
\M\F_i^{-1}\r_{i}\Biggr\rrvert ^l
\leq\frac{Cn^{l/2}\delta_n^{2l-4}}{n\th^{3l}},
\end{eqnarray*}
which completes the proof.\vadjust{\goodbreak}
\end{pf}
%
\begin{remark}
From the last lemma and \eqref{ddi}, one can easily verify that
for any $l\geq2$,
%
\begin{eqnarray}\label{efiefil}
\E\bigl|\tr\F_i^{-1}(z)\M-\E\tr\F_i^{-1}(z)
\M\bigr|^l\leq\frac
{C_ln^{l/2}\delta_n^{2l-4}}{n\th^{3l}}.
\end{eqnarray}
Furthermore, by combining \eqref{gai}, \eqref{eetail} and \eqref
{efiefil} with $\M=\I$, we have for any $l\geq2$,
%
\begin{eqnarray}\label{egail}
\E|\ga_i|^l\leq \frac{C_l\delta_n^{2l-4}}{n}.
\end{eqnarray}
\end{remark}

Denote $\S_{ij}=\S-\r_{i}\r_{i}^*-\r_j\r_j^*$ for $i\neq j$.
Correspondingly, let $\B_{ij}=\S_{ij}(\S_{ij}+\alpha_n\T_N)^{-1}$,
$\D_{ij}=\D_{ij}(z)=\B_{ij}-z\I$, $\F_{ij}=\F_{ij}(z)=({1-z})\S
_{ij}-{z\alpha}\T_N$ and assume $\|(\S_{ij}+\alpha_n\T_N)^{-1}\|
<\infty$. Moreover, we have
\begin{eqnarray*}
\vp_{ij}&=&\vp_{ij}(z)=\frac{1}{1+({1-z})\r_{j}^*\F_{ij}^{-1}(z)\r
_{j}},\qquad
\vp_{ij}^{\mathrm{tr}}=\varpi^{\mathrm{tr}}_{ij}(z)=
\frac
{1}{1+n^{-1}({1-z})\tr\F_{ij}^{-1}(z)},
\\
\vp_{ij}^{\E}&=&\varpi^{\E}_{ij}(z)=
\frac{1}{1+n^{-1}({1-z})\E_0
\tr\F_{ij}^{-1}(z)},
\\
\gamma_{ij}&=&\ga_{ij}(z)=\r_{j}^*
\F_{ij}^{-1}(z)\r_{j}-n^{-1}
\E_0 \tr\F_{ij}^{-1}(z),\\
 \eta_{ij}&=&
\eta_{ij}(z)=\r_{j}^*\F _{ij}^{-1}(z)
\r_{j}-n^{-1}\tr\F_{ij}^{-1}(z),
\\
\xi_{ij}&=&\xi_{ij}(z)=n^{-1}\tr
\F_{ij}^{-1}(z)-n^{-1} \E_0 \tr\F
_{ij}^{-1}(z).
\end{eqnarray*}
We can get the same bound as we did in \eqref{gai}--\eqref{ddi} by
changing the subscript $i$ to $ij$. Thus from now on when we consider
these bounds we will ignore the subscripts.
Let $\H_{12}=\H_{12}(z)=(1-z)\frac{n-1}{n} \varpi^{\E}_{12}\I
-z\alpha_n\T_N$. We have the following lemma.
%
\begin{lemma}\label{lemma4.3}
Under the conditions of Theorem~\ref{clt} and $z=u+\mathrm{i}\th$, we have for
any $1\leq k\leq p$, $1\leq i\leq n$ and non-random matrix $\M$ with
$\|\M\|\leq C$
%
\begin{eqnarray}\label{fiii}
\E_0\mathbf{e}_k^*\F_i^{-1}(z)
\M\mathbf{e}_k=\mathbf{e}_k^*\H_{12}^{-1}(z)
\M \mathbf{e}_k+\mathrm{O}\bigl(n^{-1/2}\bigr),
\end{eqnarray}
where $\mathbf{e}_k$ is the $p$-dimensional vector with the $k$th
coordinate being 1 and the remaining being zero.
\end{lemma}
\begin{pf}
Using
\eqref{mn}, we can check that
%
\begin{eqnarray}\label{fhhhh}
\hspace*{-10pt}\F_i^{-1}(z)&=&\H_{12}^{-1}(z) +
\frac{\varpi^{\E}_{12}(1-z)}{n }\sum_{j\neq i}\H _{12}^{-1}(z)
\bigl(\F_i^{-1}(z)-\F_{ij}^{-1}(z)
\bigr)
\nonumber
\\
\hspace*{-10pt}&&{}+\frac{\varpi^{\E}_{12}(1-z)}{n }\sum_{j\neq i}\H_{12}^{-1}(z)
\F _{ij}^{-1}(z) -(1-z)\sum_{j\neq i}
\vp_{ij}\H_{12}^{-1}(z)\r_j
\r_j^*\F _{ij}^{-1}(z)
\\
\hspace*{-10pt}&=&\H_{12}^{-1}(z) +H_{(1)}-H_{(2)}
-H_{(3)},\nonumber
\end{eqnarray}
where
\begin{eqnarray*}
H_{(1)}&=&\frac{\varpi^{\E}_{12}(1-z)}{n }\sum_{j\neq i}\H
_{12}^{-1}(z) \bigl(\F_i^{-1}(z)-
\F_{ij}^{-1}(z) \bigr),
\\
H_{(2)}&=&(1-z)\varpi^{\E}_{12}\sum
_{j\neq i} \bigl(\H _{12}^{-1}(z)
\r_j\r_j^*\F_{ij}^{-1}(z)
-n^{-1}\H_{12}^{-1}(z)\F_{ij}^{-1}(z)
\bigr),
\\
H_{(3)}&=&(1-z)\sum_{j\neq i}\bigl(
\vp_{ij}-\varpi^{\E}_{12}\bigr)\H
_{12}^{-1}(z)\r_j\r_j^*
\F_{ij}^{-1}(z).
\end{eqnarray*}
Note that, similar to \eqref{im}, either the real parts or the
imaginary parts of $(1-z)\varpi^{\E}_{12}$ and $-z $ have the same sign.
Thus, we have for any $t\geq0$
\begin{eqnarray}\label{bohij}
\biggl|(1-z)\frac{n-1}{n} \varpi^{\E}_{12}-z
\alpha_n t\biggr|^{-1}\leq\frac
{C}{\th^3},
\end{eqnarray}
which implies
%
\begin{eqnarray}\label{h12}
\bigl\llVert \H_{12}^{-1}(z)\bigr\rrVert \leq
\frac{C}{\th^3}.
\end{eqnarray}

Then it follows from \eqref{fi}, \eqref{h12} and Lemma~\ref{leby} that
\begin{eqnarray*}
&&\E_0 \bigl|\mathbf{e}_k^*\H_{12}^{-1}(z)
\r_j\r_j^*\F_{ij}^{-1}(z)\M
\mathbf{e}_k\bigr|^2
\\
&&\quad \leq Cn^{-2}\mathbf{e}_k^*\H_{12}^{-1}(z)
\H_{12}^{-1}(\bar z)\mathbf{e}_k\E
_0\mathbf{e}_k^*\M\F_{ij}^{-1}(z)
\F_{ij}^{-1}(\bar z)\M^*\mathbf{e}_k\\
&&\qquad{}+n^{-2}
\E_0\bigl|\mathbf{e}_k^*\H_{12}^{-1}(z)
\F_{ij}^{-1}(z)\M\mathbf{e}_k\bigr|^2.
\end{eqnarray*}
From \eqref{h12}, we have
%
\begin{eqnarray}\label{ekhek}
\bigl|\mathbf{e}_k^*\H_{12}^{-1}(z)
\mathbf{e}_k\bigr|\quad  \mbox{and}\quad  \mathbf{e}_k^*\H
_{12}^{-1}(z)\H_{12}^{-1}(\bar z)
\mathbf{e}_k
\end{eqnarray}
are both bounded from above. In addition, by \eqref{fi} we get that
%
\begin{eqnarray}\label{ffij}
\mathbf{e}_k^*\M\F_{ij}^{-1}(z)
\F_{ij}^{-1}(\bar z)\M\mathbf{e}_k\leq C \bigl\|
\F_{ij}^{-1}(z)\bigr\|^2\leq C\th^{-2}
\end{eqnarray}
and
%
\begin{eqnarray}\label{fijhij}
\bigl|\mathbf{e}_k^*\H_{12}^{-1}(z)
\F_{ij}^{-1}(z)\M\mathbf{e}_k\bigr|\leq C\th
^{-4}.
\end{eqnarray}
Thus combining \eqref{evpi}, \eqref{boetrvpi}, \eqref{egail}, \eqref
{ffi}, \eqref{ffij},
\eqref{fijhij} and H\"{o}lder's inequalitywe obtain
\begin{eqnarray*}
\E_0|H_{(1)}|=\mathrm{O}\bigl(n^{-1}\bigr)
\quad \mbox{and}\quad  \E_0|H_{(3)}|= \mathrm{O}\bigl(n^{-1/2}
\bigr) .
\end{eqnarray*}
Apparently we have $\E_0 H_{(2)}=0$. Thus the proof of the lemma is complete.
\end{pf}
%
\begin{lemma}\label{lemma4.4}
Under the conditions of Theorem~\ref{clt} and $z=u+\mathrm{i}\th$, we have for
any $1\leq k\leq p$, $1\leq j\leq n$ and non-random matrix $\M$ with
$\|\M\|\leq C$
\begin{eqnarray*}
\E\bigl|\mathbf{e}_k^*\F^{-1}(z)\M\mathbf{e}_k-
\E_0 \mathbf{e}_k^* \F^{-1}(z)\M
\mathbf{e}_k\bigr|^2=\mathrm{O}\bigl(n^{-1}\bigr)
\end{eqnarray*}
and
\begin{eqnarray*}
\E\bigl|\mathbf{e}_k^*\F_j^{-1}(z)\M
\mathbf{e}_k-\E_0 \mathbf{e}_k^* \F
_j^{-1}(z)\M\mathbf{e}_k\bigr|^2=
\mathrm{O}\bigl(n^{-1}\bigr).
\end{eqnarray*}
\end{lemma}
\begin{pf}
Similarly to the proof of Lemma~\ref{efef} and Lemma~\ref{lemma4.3},
we can easily get this lemma and we omit details.
\end{pf}
%
\begin{lemma}\label{lem7}
For any non-random matrix $\M$ with $\|\M\|\le C$ and $z_1=u_1+\mathrm{i}\th
_1,z_2=u_2+\mathrm{i}\th_2$ with $\min\{\th_1,\th_2\}>0$, we have
\begin{eqnarray*}
\E \biggl|\frac{1}{n}\tr\M\F_i^{-1}(z_1)
\E_i\bigl(\F_i^{-1}(z_2)\bigr)-
\E_0 \biggl(\frac{1}{n}\tr\M\F_i^{-1}(z_1)
\E_i\bigl(\F_i^{-1}(z_2)\bigr)
\biggr)\biggr |^2=\mathrm{O}\bigl(n^{-2}\bigr).
\end{eqnarray*}
\end{lemma}
%
\begin{remark}
Checking the proof of Lemma~\ref{lem7}, we see that Lemma~\ref{lem7} holds as well when we replace $\E_i(\F_i^{-1}(z_2))$ by
$\F_i^{-1}(z_2)$. The main difference in the arguments is that we do
not distinguish between the cases $j<i$ and $j>i$ when
dealing with the latter.
\end{remark}
\begin{pf*}{Proof of Lemma~\ref{lem7}} Using the martingale decomposition, we have
\begin{eqnarray*}
&&\frac{1}{n}\tr\M\F_i^{-1}(z_1)
\E_i\bigl(\F_i^{-1}(z_2)\bigr)-
\E_0 \biggl(\frac{1}{n}\tr\M\F_i^{-1}(z_1)
\E_i\bigl(\F_i^{-1}(z_2)\bigr)
\biggr)
\\
&&\quad =\frac{1}{n}\sum_{j\neq
i}^{n}(
\E_j-\E_{j-1}) \bigl[\tr\M\F_i^{-1}(z_1)
\E_i\bigl(\F_i^{-1}(z_2)\bigr)+
\tr\M\F_{ij}^{-1}(z_1) \E_i\bigl(
\F_{ij}^{-1}(z_2)\bigr) \bigr]
\\
&&\quad =\frac{1}{n}\sum_{j\neq
i}^{n}(
\E_j-\E_{j-1}) (\mathcal{K}_1+
\cK_2+\cK_3),
\end{eqnarray*}
where (via (\ref{ffi}))
\begin{eqnarray*}
\cK_1&=&\vp_{ij}(z_1)\r_j^*
\F_{ij}^{-1}(z_1) \E_i \bigl(
\vp_{ij}(z_2)\F_{ij}^{-1}(z_2)
\r_j\r_j^*\F _{ij}^{-1}(z_2)
\bigr)\M\F_{ij}^{-1}(z_1)\r_j,
\\
\cK_2&=&-\vp_{ij}(z_1)\r_j^*
\F_{ij}^{-1}(z_1) \E_i \bigl(
\F_{ij}^{-1}(z_2) \bigr)\M\F_{ij}^{-1}(z_1)
\r_j,
\\
\cK_3&=&-\tr\M\F_{ij}^{-1}(z_1)
\E_i \bigl(\vp_{ij}(z_2)\F_{ij}^{-1}(z_2)
\r_j\r_j^*\F _{ij}^{-1}(z_2)
\bigr).
\end{eqnarray*}
Note that by \eqref{ddi}
%
\begin{equation}
\label{f7} |\vp_{ij}|\bigl\|\r_j^*\F_{ij}^{-1}(z)
\bigr\|^2=\bigl|\vp_{ij}\r_j^*\F _{ij}^{-1}(z)
\F_{ij}^{-1}(\bar z)\r_j\bigr|\leq C,
\end{equation}
which implies that $\cK_1$ is bounded.

When $j>i$, applying \eqref{evpi} to get
\[
(\E_j-\E_{j-1})\cK_1=(\E_j-
\E_{j-1})\varpi^{\E}_{12}(z_1) (\cK
_{11}-\cK_{12}),
\]
where $\cK_{12}=\ga_{kj}(z_1)\cK_1$ and
\begin{eqnarray*}
\cK_{11}&=&\r_j^*\F_{ij}^{-1}(z_1)
\E_i \bigl(\vp_{ij}(z_2)\G_k(z_2)
\bigr)\M\F_{ij}^{-1}(z_1)\r_j
\\
&&{}-n^{-1}\tr\F_{ij}^{-1}(z_1)
\E_i \bigl(\vp_{ij}(z_2)\G_{ij}(z_2)
\bigr)\M\F_{ij}^{-1}(z_1)
\end{eqnarray*}
with
$\G_{ij}(z_2)=\F_{ij}^{-1}(z_2)\r_j\r_j^*\F_{ij}^{-1}(z_2)$. We
conclude from (\ref{f7}), (\ref{boetrvpi}), (\ref{egail}), Lemmas
\ref{leby}, \ref{bhi} and $\|\M\|\leq C$ that
\[
\E\Biggl|\frac{1}{n}\sum_{j>
i}^{n}(
\E_j-\E_{j-1}) (\cK_{11}-\cK_{12})\Biggr|^2
\leq \frac{M}{n^2}\sum_{j>
i}^{n}
\bigl(\E|\cK_{11}|^2+\E|\cK_{12}|^2
\bigr)\leq\frac{C}{n^2}.
\]

On the other hand, when $j<i$, we define
$\underline{\F}_{ij}^{-1}(z),\underline{\vp}_{ij}(z)$ and
$\underline{\ga}_{ij}(z)$ using
$\r_1, \ldots, \r_{j-1},\allowbreak  \underline{\r}_{j+1}, \ldots,
\underline{\r}_{i-1}, \underline{\r}_{i+1}, \ldots,
\underline{\r}_n$
as $\F_{ij}^{-1}(z),\vp_{ij}(z)$ and $\ga_{kj}(z)$ are defined
using
$\r_1, \ldots, \r_{j-1},\allowbreak  \r_{j+1}, \ldots, \r_{i-1}, \r
_{i+1},\ldots, \r_n$.
Here $\underline{\r}_{1},\dots,\underline{\r}_{n}$ are i.i.d.
copies of $\r_1$ and independent of $\{\r_j,j=1,\dots,n\}$. Let
\[
\mathcal{R}_{ij1}(z_1,z_2)=
\r_j^*\F_{ij}^{-1}(z_1) \underline{
\F}_{ij}^{-1}(z_2)\r_j,\qquad
\cR_{ij2}(z_1,z_2)=\r_j^*
\underline{\F}_{ij}^{-1}(z_2)\M
\F_{ij}^{-1}(z_1) \r_j.
\]
Applying
the equality for $\underline{\vp}_{kj}(z_2)$ similar to
(\ref{evpi}) yields
\begin{eqnarray*}
(\E_j-\E_{j-1})\cK_1&=&(\E_j-
\E_{j-1}) \bigl[\vp_{ij}(z_1)\underline {
\vp}_{ij}(z_2)\cR_{ij1}(z_1,z_2)
\cR_{ij2}(z_1,z_2) \bigr]
\\
&=&(\E_j-\E_{j-1}) (\cK_{13}+
\cK_{14}-\cK_{15}-\cK_{16}),
\end{eqnarray*}
where
\begin{eqnarray*}
\cK_{13}&=&\vp_{ij}(z_1)\underline{
\vp}_{ij}(z_2)\cT _{ij1}(z_1,z_2)
\cR_{ij2}(z_1,z_2),
\\
\cK_{14}&=&\vp_{ij}(z_1)\underline{
\vp}_{ij}(z_2)n^{-1}\tr\F_{ij}^{-1}(z_1)
\underline{\F}_{ij}^{-1}(z_2)
\cT_{ij2}(z_1,z_2),
\\
\cK_{15}&=&\varpi^{\E}_{12}(z_1)
\varpi^{\E}_{12}(z_2)\underline{\vp
}_{ij}(z_2)\underline{\ga}_{ij}(z_2)n^{-2}
\tr\F_{ij}^{-1}(z_1) \underline{
\F}_{ij}^{-1}(z_2)\tr\underline{
\F}_{ij}^{-1}(z_2)\M\F _{ij}^{-1}(z_1)
,
\\
\cK_{16}&=&\varpi^{\E}_{12}(z_1)
\vp_{ij}(z_1)\ga_{ij}(z_1)\underline
{\vp}_{ij}(z_2)n^{-2}\tr\F_{ij}^{-1}(z_1)
\underline{\F}_{ij}^{-1}(z_2)\tr\underline{
\F}_{ij}^{-1}(z_2)\M\F _{ij}^{-1}(z_1)
\end{eqnarray*}
with
\begin{eqnarray*}
\cT_{ij1}(z_1,z_2)&=&\cR_{i1}-n^{-1}
\tr\F_{ij}^{-1}(z_1) \underline{
\F}_{ij}^{-1}(z_2) ,\\
 \cT_{ij2}(z_1,z_2)&=&
\cR_{i2}-n^{-1}\tr\underline{\F}_{ij}^{-1}(z_2)
\M \F_{ij}^{-1}(z_1) .
\end{eqnarray*}
Apparently $\underline{\F}_{ij}^{-1}(z),\underline{\vp}_{ij}(z)$
and $\underline{\ga}_{ij}(z)$ have the same bound as $\F
_{ij}^{-1}(z),{\vp}_{ij}(z)$ and
${\ga}_{ij}(z)$, respectively. Thus
it follows from Lemma~\ref{leby},
(\ref{fi}) and (\ref{f7}) that
%
\begin{equation}
\label{f16} \E\bigl|\cT_{ij1}(z_1,z_2)\bigr|^2
\leq\frac{C}{n},\qquad  \E\bigl|\cT_{ij2}(z_1,z_2)\bigr|^2
\leq\frac{C}{n}
\end{equation}
and
%
\begin{eqnarray}\label{nfij2}
n^{-1}\bigl|\tr\F_{ij}^{-1}(z_1)
\underline{\F}_{ij}^{-1}(z_1)\bigr|<C,\qquad
n^{-1}\bigl|\tr\underline{\F }_{ij}^{-1}(z_2)
\M\F_{ij}^{-1}(z_1)\bigr|<C.
\end{eqnarray}
Therefore combining \eqref{egail}, \eqref{bovpi}, \eqref{boetrvpi},
(\ref{f16}) and (\ref{nfij2}), we can obtain that for $t=3,4,5,6$,
\begin{eqnarray*}
\E\bigl|(\E_j-\E_{j-1})\cK_{1t}\bigr|^2=
\mathrm{O}\bigl(n^{-1}\bigr).
\end{eqnarray*}
This via Lemma~\ref{bhi} implies that
\begin{eqnarray*}
\E\Biggl|\frac{1}{n}\sum_{j<
i}^{n}(
\E_j-\E_{j-1})\mathcal{K}_1\Biggr|^2=
\mathrm{O}\bigl(n^{-2}\bigr).
\end{eqnarray*}

The terms $\cK_2$ and $\cK_3$ can be similarly
proved to have the same order. Then the proof of Lemma~\ref{lem7} is complete.
\end{pf*}

Now we use \eqref{fhhhh} to write that
%
\begin{eqnarray}
&&\frac{1}{n}\tr\M\F_i^{-1}(z_1)
\E_i\F_i^{-1}(z_2)\nonumber\\
&&\quad =
\frac{1}{n}\tr\M \H_{12}^{-1}(z_1)
\E_i\F_i^{-1}(z_2)+
\frac{1}{n}\tr\M H_{(1)}(z_1)\E_i
\F_i^{-1}(z_2)
\\
&&\qquad {}-\frac{1}{n}\tr\M H_{(2)}(z_1)\E_i
\F_i^{-1}(z_2)-\frac{1}{n}\tr\M
H_{(3)}(z_1)\E_i\F_i^{-1}(z_2).\nonumber
\end{eqnarray}
Then we have the following lemmas.
%
\begin{lemma}\label{lemma4.7}
For any non-random matrix $\M$ with $\|\M\|\le C$ and $z_1=u_1+\mathrm{i}\th
_1,z_2=u_2+\mathrm{i}\th_2$ with $\min\{\th_1,\th_2\}>0$, we have
%
\begin{eqnarray}\label{le4.71}
\bigl|\E_0\tr\M H_{(1)}(z_1) \E_i
\bigl(\F_i^{-1}(z_2)\bigr)\bigr|=
\mathrm{O}_p(1)
\end{eqnarray}
and
%
\begin{eqnarray}\label{le4.72}
\bigl|\E_0\tr\M H_{(3)}(z_1) \E_i
\bigl(\F_i^{-1}(z_2)\bigr)\bigr|=
\mathrm{O}_p(1).
\end{eqnarray}
\end{lemma}
\begin{pf}
By \eqref{ffi}, we obtain that
\begin{eqnarray*}
&&\tr\M H_{(1)}(z_1) \E_i\bigl(
\F_i^{-1}(z_2)\bigr)
\\
&&\quad =\frac{\varpi^{\E}_{12}(z_1)(1-z_1)}{n }\sum_{j\neq i}\vp
_{ij}(z_1)\r_j^*\F_{ij}^{-1}(z_1)
\E_i\bigl(\F_i^{-1}(z_2)\bigr)\M\H
_{12}^{-1}(z_1)\F_{ij}^{-1}(z_1)
\r_j.
\end{eqnarray*}
As $\H_{12}^{-1}(z)$ $\F_i^{-1}(z)$, $\F_{ij}^{-1}(z)$, $\varpi^{\E
}_{12}(z)$ and $\vp_{12}(z)$ are all bounded when $\Im z>0 $, we can
get directly that for $j>i$,
\begin{eqnarray*}
\bigl|\E_0\tr\M H_{(1)}(z_1) \E_i
\bigl(\F_i^{-1}(z_2)\bigr)\bigr|\leq C.
\end{eqnarray*}
When $j<i$, note that we also have
\begin{eqnarray*}
\bigl|\E_0\tr\M H_{(1)}(z_1) \E_i
\bigl(\F_{ij}^{-1}(z_2)\bigr)\bigr|\leq C.
\end{eqnarray*}
Then from \eqref{ffi}, $\E|x_{ij}|<\infty$ and the definition of
$\underline{\F}_{ij}^{-1}(z),\underline{\vp}_{ij}(z)$ and
$\underline{\ga}_{ij}(z)$ in Lemma~\ref{lem7} we have
\begin{eqnarray*}
&&\bigl|\E_0 \vp_{ij}(z_1)\r_j^*
\F_{ij}^{-1}(z_1)\E_i\bigl(
\F_i^{-1}(z_2)-\F _{ij}^{-1}(z_2)
\bigr)\M\H_{12}^{-1}(z_1)\F_{ij}^{-1}(z_1)
\r_j\bigr|
\\
&&\quad =\bigl|\E_0 \vp_{ij}(z_1)\underline{
\vp}_{ij}(z_2)\r_j^*\F _{ij}^{-1}(z_1)
\underline{\F}_{ij}^{-1}(z_2)\r_j
\r_j^*\underline {\F}_{ij}^{-1}(z_2)
\M\H_{12}^{-1}(z_1)\F_{ij}^{-1}(z_1)
\r_j\bigr|=\mathrm{O}(1),
\end{eqnarray*}
which completes the proof of \eqref{le4.71}.

Now consider \eqref{le4.72}. When $j<i$, using \eqref{evpi} we
rewrite the left-hand side of \eqref{le4.72} as
%
\begin{eqnarray}\label{le4.73}
&&\Biggl|(1-z_1) \varpi^{\E}_{12}(z_1)
\E_0 \sum_{j\neq i}\vp_{ij}(z_1)
\ga _{ij}(z_1)\r_j^*\F_{ij}^{-1}(z_1)
\E_i\bigl(\F_i^{-1}(z_2)\bigr)\M
\H_{12}^{-1}(z_1)\r_j\Biggr|
\nonumber
\\[-8pt]\\[-8pt]
&&\quad = \Biggl|(1-z_1) \varpi^{\E}_{12}(z_1)
\E_0 \sum_{j\neq i}\vp _{ij}(z_1)
\ga_{ij}(z_1)\cT_{ij3}(z_1,z_2)\nonumber
\\
\label{le4.74}&&\hphantom{\quad = |}{}+(1-z_1) \varpi^{\E}_{12}(z_1)
\E_0 n^{-1}\nonumber\\[-8pt]\\[-8pt]
&&\hphantom{|=|+(}\times\sum_{j\neq i}\vp
_{ij}(z_1)\ga_{ij}(z_1)\tr
\F_{ij}^{-1}(z_1) \E_i\bigl(
\F_i^{-1}(z_2)\bigr)\M\H_{12}^{-1}(z_1)
\Biggr|,\nonumber
\end{eqnarray}
where
\[
\cT_{ij3}(z_1,z_2)=\r_j^*
\F_{ij}^{-1}(z_1) \E_i\bigl(
\F_i^{-1}(z_2)\bigr)\M\H_{12}^{-1}(z_1)
\r_j-n^{-1}\tr\F_{ij}^{-1}(z_1)
\E_i\bigl(\F_i^{-1}(z_2)\bigr)\M
\H_{12}^{-1}(z_1).
\]
From Lemma~\ref{leby}, we have $\E|\cT_{ij3}(z_1,z_2)|^2=\mathrm{O}(n^{-1})$
which together with \eqref{egail} and H\"{o}lder's inequalityimplies
\begin{eqnarray*}
\eqref{le4.73}=\mathrm{O}(1).
\end{eqnarray*}
For \eqref{le4.74}, we apply \eqref{evpi} again and obtain that
\begin{eqnarray*}
\bigl|\eqref{le4.74}\bigr|=\Biggl| \bigl((1-z_1) \varpi^{\E}_{12}(z_1)
\bigr)^2\E _0 n^{-1}\sum
_{j\neq i}\vp_{ij}(z_1)
\ga_{ij}^2(z_1)\tr\F_{ij}^{-1}(z_1)
\E_i\bigl(\F_i^{-1}(z_2)\bigr)\M
\H_{12}^{-1}(z_1)\Biggr|.
\end{eqnarray*}
Here we have used the fact that $|n^{-1}\tr\F_{ij}^{-1}(z_1)
\E_i(\F_i^{-1}(z_2))\M\H_{12}^{-1}(z_1)|$ is bounded. Thus from
\eqref{egail}, we get that
\begin{eqnarray*}
\eqref{le4.74}=\mathrm{O}(1).
\end{eqnarray*}

On the other hand, when $j>i$, the above argument apparently also works
if we replace $\E_i(\F_i^{-1}(z_2))$ with $\E_i(\F
_{ij}^{-1}(z_2))$. And the remaining term can be expressed as
\begin{eqnarray}\label
{le4.7x}
&&(1-z_1) \varpi^{\E}_{12}(z_1)\E
\sum_{j\neq i}\vp _{ij}(z_1)
\underline\vp_{ij}(z_2)\ga_{ij}(z_1)
\r_j^*\F_{ij}^{-1}(z_1) \underline{
\F}_{ij}^{-1}(z_2)\r_j
\r_j^*\underline{\F }_{ij}^{-1}(z_2)
\M\H_{12}^{-1}(z_1)\r_j
\nonumber\\
&&\quad =(1-z_1) \varpi^{\E}_{12}(z_1)\E
\sum_{j\neq i}\vp _{ij}(z_1)
\underline\vp_{ij}(z_2)\ga_{ij}(z_1)
\cT_{ij1}\r _j^*\underline{\F}_{ij}^{-1}(z_2)
\M\H_{12}^{-1}(z_1)\r_j
\\
&&\qquad {}+(1-z_1) \varpi^{\E}_{12}(z_1)\E
n^{-1}\sum_{j\neq i}\vp _{ij}(z_1)
\underline\vp_{ij}(z_2)\ga_{ij}(z_1)
\tr\F_{ij}^{-1}(z_1) \underline{
\F}_{ij}^{-1}(z_2)\cT_{ij2}(z_1,z_2)
\nonumber\\
&&\qquad {}-(1-z_1) (1-z_2) \varpi^{\E}_{12}(z_1)
\varpi^{\E}_{12}(z_2)\E n^{-2}\sum
_{j\neq i}\vp_{ij}(z_1)\underline
\vp_{ij}(z_2)\ga _{ij}(z_1)
\ga_{ij}(z_2)
\nonumber\\
&&\hphantom{\qquad {}-(1-z_1) (1-z_2) \varpi^{\E}_{12}(z_1)
\varpi^{\E}_{12}(z_2)\E n^{-2}\sum
_{j\neq i}}{}\cdot\tr\F_{ij}^{-1}(z_1) \underline{
\F}_{ij}^{-1}(z_2)\tr\underline{
\F}_{ij}^{-1}(z_2)\M\H _{12}^{-1}(z_1)
\nonumber\\
\label{le4.7y}
&&\qquad {}-\bigl((1-z_1) \varpi^{\E}_{12}(z_1)
\bigr)^2\E n^{-2}\sum_{j\neq i}\vp
_{ij}(z_1)\underline\vp_{ij}(z_2)
\bigl(\ga_{ij}(z_1) \bigr)^2
\\
&&\hphantom{\qquad {}-\bigl((1-z_1) \varpi^{\E}_{12}(z_1)
\bigr)^2\E n^{-2}\sum_{j\neq i}}{}\cdot\tr\F_{ij}^{-1}(z_1) \underline{
\F}_{ij}^{-1}(z_2)\tr\underline{
\F}_{ij}^{-1}(z_2)\M\H _{12}^{-1}(z_1)
\nonumber
.
\end{eqnarray}
By Lemma~\ref{leby} and a similar argument in \eqref{le4.71}, we can
show that \eqref{le4.7x} and \eqref{le4.7y} are all bounded. Then the
proof of Lemma~\ref{lemma4.7} is complete.
\end{pf}
%
\begin{lemma}
\label{lemma4.8}
For any non-random matrix $\M$ with $\|\M\|\le C$ and $z_1=u_1+\mathrm{i}\th
_1,z_2=u_2+\mathrm{i}\th_2$ with $\min\{\th_1,\th_2\}>0$, we have
\begin{eqnarray*}
&&\E_0 n^{-1}\tr\M H_{(2)}(z_1)
\E_i\bigl(\F_i^{-1}(z_2)\bigr)
\nonumber
\\
&&\quad =-(i-1)n^{-2}(1-z_1) (1-z_2)
\varpi^{\E}_{12}(z_1)\varpi^{\E
}_{12}(z_2)
\\
&&\qquad {}\cdot\tr\H_{12}^{-1}(z_2)\M
\H_{12}^{-1}(z_1)\E_0\tr\F
_{i}^{-1}(z_1)\underline{\F}_{i}^{-1}(z_2)+
\mathrm{O}_p(1).
\nonumber
\end{eqnarray*}
\end{lemma}
\begin{pf}
It follows from \eqref{ffi}, \eqref{boetrvpi}, \eqref{fi}, \eqref
{h12} and $\E|x_{ij}|<C$ that
%
\begin{eqnarray}\label{le4.81}
&&\E_0\tr\M H_{(2)}(z_1) \E_i
\bigl(\F_i^{-1}(z_2)\bigr)
\nonumber
\\
&&\quad =-(1-z_1) (1-z_2)\varpi^{\E}_{12}(z_1)\nonumber
\\
&&\qquad {}\times\sum_{j< i}\E_0\underline\vp
_{ij}(z_2)\r_j^*\F_{ij}^{-1}(z)
\underline{\F}_{ij}^{-1}(z_2)\r_j\r
_j^*\underline{\F}_{ij}^{-1}(z_2)\M
\H_{12}^{-1}(z_1)\r_j
\\
&&\qquad {}+(1-z_1) (1-z_2)\varpi^{\E}_{12}(z_1)\nonumber\\
&&\quad \qquad {}\times
\frac{1}{n}\sum_{j< i}\E _0
\underline\vp_{ij}(z_2)\r_j^*\underline{
\F}_{ij}^{-1}(z_2)\M\H _{12}^{-1}(z_1)
\F_{ij}^{-1}(z)\underline{\F}_{ij}^{-1}(z_2)
\r _j\nonumber
\end{eqnarray}
and
\begin{eqnarray*}
(1-z_1) (1-z_2)\varpi^{\E}_{12}(z_1)
\sum_{j< i}\E_0\underline\vp
_{ij}(z_2)\r_j^*\underline{
\F}_{ij}^{-1}(z_2)\M\H_{12}^{-1}(z_1)
\F _{ij}^{-1}(z)\underline{\F}_{ij}^{-1}(z_2)
\r_j=\mathrm{O}_p(n).
\end{eqnarray*}
Applying \eqref{evpi} to rewrite the first term of \eqref{le4.81} as
%
\begin{eqnarray}
&&\sum_{j< i}\E_0\underline
\vp_{ij}(z_2)\r_j^*\F _{ij}^{-1}(z)
\underline{\F}_{ij}^{-1}(z_2)\r_j
\r_j^*\underline{\F }_{ij}^{-1}(z_2)
\M\H_{12}^{-1}(z_1)\r_j
\nonumber
\\
&&\quad =\varpi^{\E}_{12}(z_2)n^{-2}\sum
_{j< i}\E_0\tr\F _{ij}^{-1}(z_1)
\underline{\F}_{ij}^{-1}(z_2)\tr\underline{\F
}_{ij}^{-1}(z_2)\M\H_{12}^{-1}(z_1)\label{le4.82}
\\
&&\qquad {}+\varpi^{\E}_{12}(z_2)\sum
_{j< i}\E_0\cT_{ij1}(z_1,z_2)
\cT _{ij4}(z_1,z_2)\label{le4.83}
\\
&&\qquad {}-\varpi^{\E}_{12}(z_2)\sum
_{j< i}\E_0\underline\vp _{ij}(z_2)
\underline\ga_{ij}(z_2)\r_j^*
\F_{ij}^{-1}(z)\underline {\F}_{ij}^{-1}(z_2)
\r_j\r_j^*\underline{\F}_{ij}^{-1}(z_2)
\H _{12}^{-1}(z_1)\r_j\label{le4.84},
\end{eqnarray}
where $\cT_{ij4}(z_1,z_2)=\r_j^*\underline{\F}_{ij}^{-1}(z_2)\M\H
_{12}^{-1}(z_1)\r_j-n^{-1}\tr\underline{\F}_{ij}^{-1}(z_2)\M\H
_{12}^{-1}(z_1)$.
The arguments\vspace{2pt} in \eqref{le4.7x} and \eqref{le4.7y} and \eqref{f16}
ensure that
\begin{eqnarray*}
\eqref{le4.83}=\mathrm{O}_p(1)\quad  \mbox{and}\quad  \eqref{le4.84}=
\mathrm{O}_p(1).
\end{eqnarray*}
In addition, from \eqref{ffi} and \eqref{ddi}, we have
\begin{eqnarray*}
\E_0\tr\underline{\F}_{ij}^{-1}(z_2)
\H_{12}^{-1}(z_1)=\E_0 \tr
\underline{\F}_{i}^{-1}(z_2)\H_{12}^{-1}(z_1)+
\mathrm{O}_p(1).
\end{eqnarray*}
Then using \eqref{fhhhh} again and repeating similar arguments as in
the proof of Lemma~\ref{lemma4.7} we obtain that
\begin{eqnarray*}
\E_0 \tr\underline{\F}_{i}^{-1}(z_2)
\H_{12}^{-1}(z_1)=\tr\H _{12}^{-1}(z_2)
\H_{12}^{-1}(z_1)+\mathrm{O}_p(1).
\end{eqnarray*}
Combining the above arguments, we conclude that
\begin{eqnarray*}
&&\E_0 n^{-1}\tr\M H_{(2)}(z_1)
\E_i\bigl(\F_i^{-1}(z_2)\bigr)
\\
&&\quad =-(i-1)n^{-2}(1-z_1) (1-z_2)
\varpi^{\E}_{12}(z_1)\varpi^{\E
}_{12}(z_2)\\
&&\qquad{}\times\tr\H_{12}^{-1}(z_2)\M
\H_{12}^{-1}(z_1)\E_0\tr\F
_{i}^{-1}(z_1)\underline{\F}_{i}^{-1}(z_2)\\
&&\qquad {}+
\mathrm{O}_p(1),
\end{eqnarray*}
which complete the proof of the lemma.
\end{pf}
%
\begin{remark}
Let $\H_{1}=\H_{1}(z)=(1-z) \varpi^{\E}_{1}\I-z\alpha_n\T_N$. We
conclude from the above arguments
and the fact $|\varpi^{\E}_{12}-\varpi^{\E}_1|=\mathrm{O}(n^{-1})$ that\vspace*{1pt}
%
\begin{eqnarray}
\label{good1} \mathbf{e}_k^*\F_i^{-1}(z)
\mathbf{e}_k=\mathbf{e}_k^*\H_{1}^{-1}(z)
\mathbf{e}_k+\mathrm{O}_p\bigl(n^{-1/2}\bigr)
\end{eqnarray}
and\vspace*{1pt}
\begin{eqnarray}\label{good4}
&&\frac{1}{n}\tr\F_i^{-1}(z_1)
\E_i\bigl(\F_i^{-1}(z_2)\bigr)\nonumber\\[-8pt]\\[-8pt]
&&\quad \to
\frac{\sklfrac{1}{n}\tr\H_{1}^{-1}(z_2)\H
_{1}^{-1}(z_1)}{1-\sklvfrac{(i-1)(1-z_1)(1-z_2)\varpi^{\E
}_{1}(z_1)\varpi^{\E}_{1}(z_2)}{n^{2}}\tr\H_{1}^{-1}(z_2)\H
_{1}^{-1}(z_1)}.\nonumber
\end{eqnarray}
Here we have used the fact that the denominator
of \eqref{good4} is bounded when $\min\{\th_1,\th_2\}>0$.
\end{remark}

%

%

\subsection{Proof of Lemma \texorpdfstring{\protect\ref{leclt}}{3.1}}
Note that
the contour $\cC$ of the integration
contains four segments: two horizontal lines and two vertical lines. We
need to calculate the limit of $S_{n1}(z)$ at the four segments respectively.
First of all, considering the top horizontal line $\cC^t=\{z\in
\mathbb{C}:\Re z\in[c_l-\th,c_r+\th],\Im z=\th\}$, we know that
there exists some event $\mathcal{Q}_n$ with $\P(\cQ_n)\to1$ such that,\vspace*{1pt}
\begin{eqnarray*}
\E\bigl|s_n(z) -s_n(z)I(\cQ_n)\bigr|\leq(\Im
z)^{-1}\P\bigl(\cQ_n^c\bigr)\to0.
\end{eqnarray*}
In this part, we let $\cQ=\cQ_n=\{\|(\S_n+\alpha_n\T_N)^{-1}\|\leq
C\}$ with some $C<\infty$. By \eqref{rem1.3} we have that for any
$l>0$, $\P(\cQ^c)\leq n^{-l}$. It is known that $\la_1^{\S+\alpha
_n\T_N}\geq\la_1^{\S_{i}+\alpha_n\T_N}\geq\la_1^{\S
_{ij}+\alpha_n\T_N}$ for any $i,j$, which implies\vspace*{1pt}
\begin{eqnarray*}
\cQ\supseteq\cQ_{i}\supseteq\cQ_{ij}.
\end{eqnarray*}
Here $\cQ_{i}=\{\|(\S_{i}+\alpha_n\T_N)^{-1}\|\le C\}$ and $\cQ
_{ij}=\{\|(\S_{ij}+\alpha_n\T_N)^{-1}\|\leq C\}$. Notice that we
also have
\begin{eqnarray*}
\P\bigl(\cQ_i^c\bigr)\leq n^{-l}\quad \mbox{and}\quad
\P\bigl(\cQ_{ij}^c\bigr)\leq n^{-l}.
\end{eqnarray*}
Now
we rewrite $S_{n1}(z)$ as $S_{n1}=S_{n1}^{(1)}+S_{n1}^{(2)}+\mathrm{o}_p(1)$ with
\begin{eqnarray*}
S_{n1}^{(1)}&=&p \bigl(s_n(z)I(\cQ)-
\E_0\bigl[ s_n(z)I(\cQ)\bigr] \bigr)\qquad  \mbox {covariance
part},
\\
S_{n1}^{(2)}&=&p \bigl(\E_0 s_n(z)I(
\cQ)-s_0(z)I(\cQ) \bigr) \qquad \mbox {mean part}.
\end{eqnarray*}
%
\subsubsection{The covariance part}
The martingale decomposition used in the proof of Lemma~\ref{efef}
gives that
\begin{eqnarray*}
S_{n1}^{(1)}&=&\sum_{i=1}^n(
\E_i-\E_{i-1})\tr\bigl(\D^{-1}-\D
_i^{-1}\bigr)I(\cQ_i)+\mathrm{o}_p(1)
\\
&=&\sum_{i=1}^n(\E_i-
\E_{i-1})\tr(\S-\S_i)\F^{-1}(z)I(
\cQ_i)
\\
&&{} +\sum_{i=1}^n(\E_i-
\E_{i-1})\tr(\S_i+\alpha_n\T_N)
\bigl(\F ^{-1}-\F_i^{-1} \bigr)I(
\cQ_i)+\mathrm{o}_p(1)
\\
&=&\sum_{i=1}^n(\E_i-
\E_{i-1}) \eta_i(z)I(\cQ_i)-\mathcal
D_1-\cD_2+\mathrm{o}_p(1),
\end{eqnarray*}
where
\begin{eqnarray*}
\cD_1&=&\sum_{i=1}^n(
\E_i-\E_{i-1}) (1-z)\varpi_i
\r_{i}^*\F _i^{-1}(z)\r_i
\r_i^*\F_i^{-1}(z)\r_{i}I(
\cQ_i),
\\
\cD_2&=&\sum_{i=1}^n(
\E_i-\E_{i-1}) (1-z)\varpi_i
\r_{i}^*\F _i^{-1}(z) (\S_i+
\alpha_n\T_N)\F_i^{-1}(z)
\r_{i}I(\cQ_i).
\end{eqnarray*}
Here we used (\ref{ffi}) and the fact that
%
\begin{eqnarray}\label{pcqicq}
\P\bigl(I(\cQ_i)\neq I(\cQ)\bigr)\leq n^{-l}.
\end{eqnarray}
Check that
\begin{eqnarray*}
\r_{i}^*\F_i^{-1}(z)\r_i
\r_i^*\F_i^{-1}(z)\r_{i}=\eta
_i^2(z)+\frac{2}{n}\eta_i(z)\tr
\F_i^{-1}(z)+\biggl(\frac{1}{n}\tr
\F_i^{-1}(z)\biggr)^2.
\end{eqnarray*}
Applying \eqref{trvpi}, \eqref{eetail} and Lemma~\ref{bhi} we obtain
\begin{eqnarray*}
\E\biggl| \cD_1-\sum_{i=1}^n(
\E_i-\E_{i-1}) \biggl( \frac{2(1-z) \varpi
^{\mathrm{tr}}_i\eta_i}{n}\tr
\F_i^{-1}-\frac{( 1-z)^2(\varpi^{\mathrm{tr}}_i)^2\eta
_i}{n^2}\bigl(\tr\F_i^{-1}
\bigr)^2\biggr)I(\cQ_i)\biggr|^2=\mathrm{o}(1).
\end{eqnarray*}
Similarly we have
\begin{eqnarray*}
&&\E\biggl| \cD_2-\sum_{i=1}^n(
\E_i-\E_{i-1}) \biggl((1-z)\varpi^{\mathrm{tr}}_iK_i(z)\\
&&\hphantom{\E\biggl| \cD_2-\sum_{i=1}^n(
\E_i-\E_{i-1}) \biggl(}{}-
\frac{(1-z)^2(\varpi^{\mathrm{tr}}_i)^2\eta_i}{n}\tr\F_i^{-2}(\S_i+\alpha
_n\T_N) \biggr)I(\cQ_i)\biggr|^2=
\mathrm{o}(1),
\end{eqnarray*}
where $K_i(z)=\r_{i}^*\F_i^{-1}(z)(\S_i+\alpha_n\T_N)\F
_i^{-1}(z)\r_{i}-n^{-1}\tr\F_i^{-2}(z)(\S_i+\alpha_n\T_N)$.
Thus we have
\begin{eqnarray*}
&&p\bigl(s_n(z)-\E s_n(z)\bigr)I(\cQ)
\\
&&\quad =\sum_{i=1}^n
\E_i \biggl( \bigl(\varpi^{\mathrm{tr}}_i
\bigr)^2\eta _i-(1-z)\varpi^{\mathrm{tr}}_iK_i(z)
\\
&&\hphantom{\sum_{i=1}^n
\E_i \biggl(}\qquad {} + \frac{(1-z)^2(\varpi^{\mathrm{tr}}_i)^2\eta_i}{n}\tr\F _i^{-1}(z) (
\S_i+\alpha_n\T_N)\F_i^{-1}(z)
\biggr)I(\cQ_i)+\mathrm{o}_p(1).
\end{eqnarray*}
Check that
\begin{eqnarray*}
-\frac{\mathrm{d}(1-z)\varpi^{\mathrm{tr}}_i(z)\eta_i(z)}{\mathrm{d}z} 
&=&-(1-z)
\varpi^{\mathrm{tr}}_i(z)K_i(z)+ \bigl(
\varpi^{\mathrm{tr}}_i(z)\bigr)^2\eta_i(z)
\\
&&{}+\frac{(1-z)^2(\varpi
^{\mathrm{tr}}_i(z))^2\eta_i(z)}{n}\tr\F_i^{-2}(z) (\S_i+
\alpha_n\T_N),
\end{eqnarray*}
which implies
\begin{eqnarray*}
&&\frac{1}{2\uppi \mathrm{i}}\int_{\mathcal{C}^t} f(z)p\bigl(s_n(z)-
\E s_n(z)\bigr)I(\cQ _i)\,\mathrm{d}z
\\
&&\quad =-\frac{1}{2\uppi \mathrm{i}}\sum_{i=1}^n\int
_{\mathcal{C}^t} f(z)\E_i I(\cQ _i) \,\mathrm{d}(1-z)
\varpi^{\mathrm{tr}}_i(z)\eta_i(z)+
\mathrm{o}_p(1).
\end{eqnarray*}

Apparently, $\{\E_i I(\cQ_i)\,\mathrm{d}(1-z)\varpi^{\mathrm{tr}}_i(z)\eta_i(z)/\mathrm{d}z\}$
is a martingale difference sequence so we can
resort to the CLT for martingale (see Theorem~35.12 in \cite
{Billingsley95P}). By Lemma~\ref{leby} and \eqref{fi}, we can get
\begin{eqnarray*}
\E\bigl|K_i(z)I(\cQ_i)\bigr|^4\leq
\frac{C\delta_n^4}{n},
\end{eqnarray*}
which together with \eqref{eetail} and \eqref{boetrvpi} implies
\[
\sum_{k=1}^n\E\bigl|I(\cQ_i)
\,\mathrm{d}(1-z)\varpi^{\mathrm{tr}}_i(z)\eta_i(z)/
\mathrm{d}z\bigr|^4= \mathrm{O}(\delta_n) \to0.
\]
This ensures the Lyapunov condition.
Thus, it is sufficient to investigate the limit of the following
covariance function
%
\begin{eqnarray}
-\frac{1}{4\uppi^2}\int_{\cC_1^t}\int_{\cC_2^t}
f(z_1) f(z_2)\frac{\partial^2}{\partial z_1\,\partial z_2}\mathcal
{G}_n(z_1,z_2)\,\mathrm{d}z_1\,
\mathrm{d}z_2,
\end{eqnarray}
where
\[
\mathcal{G}_{n}(z_1,z_2)=\sum
_{i=1}^n\E_{i-1} [\E_i
\bigl((1-z_1)\vp^{\mathrm{tr}}_i(z_1)
\eta_i(z_1)I(\cQ_i) \bigr)\E_i
\bigl((1-z_2)\vp^{\mathrm{tr}}_i(z_2)
\eta_i(z_2)I(\cQ_i) \bigr).
\]
From the arguments in \cite{BaiS04C}, we need to show $\mathcal
{G}_{n}(z_1,z_2)$ converges in probability. Applying (\ref{boetrvpi}),
\eqref{eetail}, \eqref{efiefil} and the fact $\varpi^{\mathrm{tr}}_i=\varpi
^{\E}_i-\varpi^{\mathrm{tr}}_i\varpi^{\E}_i\xi_i$, we have\vspace*{-1pt}
\begin{eqnarray*}
&&\mathcal{G}_{n}(z_1,z_2)
\\
&&\quad =(1-z_1) (1-z_2)\sum_{i=1}^n
\vp^{\E}_1(z_1)\vp^{\E}_1(z_2)
\E _{i-1} \bigl[\E_i \bigl(\eta_i(z_1)I(
\cQ_i) \bigr)\E_i \bigl(\eta _i(z_2)I(
\cQ_i) \bigr) \bigr]+\mathrm{o}_p(1).
\end{eqnarray*}
By Lemma~\ref{le5.2},
we have
%
\begin{eqnarray}
&&\E_{i-1} \bigl[\E_i \bigl(\eta_i(z_1)I(
\cQ_i) \bigr)\E_i \bigl(\eta _i(z_2)I(
\cQ_i) \bigr) \bigr]
\nonumber
\\[-8pt]\\[-8pt]
\label{ff_1}&&\quad =\frac{\E|x_{11}|^4-|\E x_{11}^2|-2}{n^2}\E_{i-1}\sum_{j=1}^n
\bigl[\E_i \bigl(\F_i^{-1}(z_1)I(
\cQ_i) \bigr)_{jj}\E_i \bigl(\F
_i^{-1}(z_2)I(\cQ_i)
\bigr)_{jj} \bigr]\nonumber
\\
\label{ff_2}&&\qquad {}+\frac{\E x_{11}^2+1}{n^2}\E_{i-1}\tr \bigl[\E_i \bigl(\F
_i^{-1}(z_1)I(\cQ_i) \bigr)
\E_i \bigl(\F_i^{-1}(z_2)I(
\cQ_i) \bigr) \bigr].
\end{eqnarray}
Using \eqref{good1}, we have
\begin{eqnarray*}
\eqref{ff_1}=\frac{\frm_x-\frt-2}{n^2}\sum_{j=1}^n
\bigl[ \bigl(\H _1^{-1}(z_1)
\bigr)_{jj} \bigl(\H_1^{-1}(z_2)
\bigr)_{jj} \bigr]+\mathrm{o}_p(1).
\end{eqnarray*}
It is worthy to remind the reader that in order to satisfy the
condition in the last subsection we used here the fact
\begin{eqnarray*}
\P\bigl(I(\cQ_i)\neq I(\cQ_{ij})\bigr)\leq
n^{-l}.
\end{eqnarray*}
And by \eqref{good4}, we have\vspace*{-1pt}
\begin{eqnarray*}
\eqref{ff_2}=\frac{\frt+1}{n}\frac{\sklfrac{1}{n}\tr\H
_{1}^{-1}(z_2)\H_{1}^{-1}(z_1)}{1-\sklvfrac{(i-1)(1-z_1)(1-z_2)\varpi
^{\E}_{1}(z_1)\varpi^{\E}_{1}(z_2)}{n^{2}}\tr\H_{1}^{-1}(z_2)\H
_{1}^{-1}(z_1)}+
\mathrm{o}_p(1).
\end{eqnarray*}
From the arguments of the next part, we can conclude that for $z\in\cC^t$
\begin{eqnarray*}
\E_0 s_n(z)=s_0(z)+\mathrm{O}
\bigl(n^{-1}\bigr)\stackrel{\mathrm{i.p.}} {\longrightarrow} s(z).
\end{eqnarray*}
Thus we get in probability
\begin{eqnarray*}
&&\mathcal{G}_{n}(z_1,z_2)\\
&&\quad \to(\frt+1)\int
\biggl[\biggl(\int\frac
{y(1-z_1)(1-z_2)\vp(z_1)\vp(z_2)}{ ((1-z_1)\vp-z_1\alpha
t ) ((1-z_2)\vp-z_2\alpha t )}\,\mathrm{d}F_{\mathrm{mp}}^{Y}(t)\biggr)\\
&&\hphantom{\quad \to(\frt+1)\int
\biggl[}{}\times \biggl(1-w\int
\frac{y(1-z_1)(1-z_2)\vp(z_1)\vp(z_2)}{ ((1-z_1)\vp-z_1\alpha
t ) ((1-z_2)\vp-z_2\alpha t
)}\,\mathrm{d}F_{\mathrm{mp}}^{Y}(t)\biggr)^{-1}\biggr]\,\mathrm{d}w
\nonumber
\\
&&\qquad\,\,{}{}+(\frm_x-\frt-2)y\int\frac{(1-z_1)\vp(z_1)}{(1-z_1)\vp-z_1\alpha
t}\,\mathrm{d}F_{\mathrm{mp}}^{Y}(t)
\int\frac{(1-z_2)\vp(z_2)}{(1-z_2)\vp-z_2\alpha
t}\,\mathrm{d}F_{\mathrm{mp}}^{Y}(t),
\end{eqnarray*}
which is \eqref{phi4}.

In addition,
by
definition of $S_{n1}^{(1)}$ we get
\begin{eqnarray*}
\E\bigl|S_{n1}^{(1)}(z_1)-S_{n1}^{(1)}(z_2)\bigr|^2=|z_1-z_2|^2
\E\bigl|\tr\D ^{-1}(z_1)\D^{-1}(z_2)-
\E_0\tr\D^{-1}(z_1)\D^{-1}(z_2)\bigr|^2I(
\cQ). 
\end{eqnarray*}
Therefore using \eqref{ffi}, Lemmas \ref{efef}, \ref{lem7} and the
fact that
\[
\D^{-1}(z)=(1-z)^{-1}\bigl(\I+\alpha_n
\T_N\F^{-1}(z)\bigr),
\]
we can easily check that
%
\begin{eqnarray}\label{tight}
\E\bigl|S_{n1}^{(1)}(z_1)-S_{n1}^{(1)}(z_2)|^2
\leq C |z_1-z_2\bigr|^2,\qquad  z_1,z_2
\in\cC^t,
\end{eqnarray}
which implies the sequence $\{S_{n1}^{(1)}(\cdot)\}$ forms
a tight sequence on $\cC^t$.

\subsubsection{The mean part}
From the definition of the Stieltje transform of $s_n(z)$, we have
\begin{eqnarray*}
s_n(z)&=&s_{F^{\B_n}}=\frac{1}{p}\tr\D^{-1}=
\frac{1}{p}\tr(\S _n+\alpha_n\T_N)
\F^{-1}(z)
\\
&=&\biggl(1+\frac{1-z}{z}\biggr)\frac{1}{p}\tr\S_n
\F^{-1}(z)-\frac{1}{z} =\frac{1}{zp}\tr\S_n
\F^{-1}(z)-\frac{1}{z}.
\end{eqnarray*}
Using \eqref{mn}
we get that
%
\begin{eqnarray}\label{sn1}
\S_n\F^{-1}(z) =\sum_{i=1}^n
\varpi_i\r_{i}\r_{i}^*\F_i^{-1},
\end{eqnarray}
which implies
\begin{eqnarray*}
\frac{1}{p}\tr\S_n\F^{-1}(z) 
=
\frac{n}{p(1-z)} \Biggl(1-\frac{1}{n}\sum_{i=1}^n
\vp_i \Biggr).
\end{eqnarray*}
Thus we have
%
\begin{eqnarray}
\frac{1}{n}\sum_{i=1}^n
\vp_i=1-y_n(1-z) (zs_n+1)
\end{eqnarray}
and
%
\begin{eqnarray}\label{evp1}
\E_0\vp_1I(\cQ_1)=1-y_n(1-z)
\bigl(z\E_0 s_nI(\cQ_1)+1
\bigr).
\end{eqnarray}

Denote $\A_n=
\E_0(\vp_1I(\cQ_1)) (\E_0(\vp_1I(\cQ_1))\I+\alpha_n\T
_N )^{-1}$,
\[
\C_n=\A_n-z\I \De(z)=\E_0\quad  \mbox{and}\quad
s_n(z)I(\cQ_1)-p^{-1}\tr\C_n.
\]
Then we obtain that
%
\begin{eqnarray}
\label{eqcn} p^{-1}\tr\C_n =\int\frac{\E_0\vp_1I(\cQ_1)+\alpha_nt}{(1-z)\E_0\vp_1I(\cQ
_1)-z\alpha_nt}\,
\mathrm{d}F^{\T_N}(t).
\end{eqnarray}

Recalling the definition of $\vp_0$ and \eqref{sti1} we have
\begin{eqnarray*}
\frac{1-\vp_0}{zy(1-z)}-\frac{1}{z}=\frac{1}{1-z}+\frac
{1}{1-z}
\int\frac{\alpha_nt}{(1-z)\vp_0-z\alpha_nt}\,\mathrm{d}F^{\T_N}(t),
\end{eqnarray*}
which implies
\begin{eqnarray*}
\vp_0= \biggl(1+y\int\frac{(1-z)}{(1-z)\vp_0-z\alpha_nt}\,\mathrm{d}F^{\T
_N}(t)
\biggr)^{-1}.
\end{eqnarray*}
According to \eqref{evp1} and \eqref{eqcn}, we get that
\begin{eqnarray*}
&&\E_0\vp_1I(\cQ_1)
\\
&&\quad =\biggl(1+y\int\frac{(1-z)}{(1-z)\E_0\vp_1I(\cQ
_1)-z\alpha_nt}\,\mathrm{d}F^{\T_N}(t)+\bigl(
\E_0\vp_1I(\cQ_1)\bigr)^{-1}zy(1-z)
\De_n\biggr)^{-1}.
\end{eqnarray*}
The difference of the above two identities yields
\begin{eqnarray*}
\vp_0-\E\vp_1I(\cQ_1) 
&= &\int
\frac{\E_0 \vp_1I(\cQ_1)\vp_0 y(1-z)^2 (\vp_0-\E
_0\vp_1I(\cQ_1) )}{[(1-z)\vp_0-z\alpha_nt][(1-z)\E_0\vp
_1I(\cQ_1)-z\alpha_nt]}\,\mathrm{d}F^{\T_N}(t)
\\
&&{}+\vp_0zy(1-z)\De_n.
\end{eqnarray*}
Thus we use \eqref{pcqicq} to obtain that
%
\begin{eqnarray}\label{last1}
&&\hspace*{-20pt}\E_0 s_n(z)I(\cQ)-s_0(z)
\nonumber
\\[-8pt]\\[-8pt]
&&\hspace*{-20pt}\quad =\vp_0 \De_n \biggl(1-\int\frac{y_n(1-z)^2\E_0 \vp_1I(\cQ_1)\vp
_0}{[(1-z)\vp_0-z\alpha_nt][(1-z)\E_0\vp_1I(\cQ_1)-z\alpha
_nt]}\,
\mathrm{d}F^{\T_N}(t) \biggr)^{-1}.\nonumber
\end{eqnarray}

We will use the following lemma.
%
\begin{lemma} \label{lede}For $z\in\cC_t$
\begin{eqnarray*}
p\De(z)&=&\frac{(\frm_x-\frt-2)\alpha_n(1-z)(\varpi^{\E
}_1)^2}{pn}\tr\H_0^{-1}(z)\tr
\H^{-2}_0(z)\T_N
\\
&&{}+\frac{\frt(1-z)(\varpi^{\E}_1)^2\alpha_n }{n}\tr\H_0^{-3}(z)\T _N +
\mathrm{o}(1).
\end{eqnarray*}
\end{lemma}
\begin{pf}
It follows from the definition of $\D_n$ and $\C_n$
that
\begin{eqnarray*}
\D_n^{-1}-\C_n^{-1}=
\C_n^{-1}(\A_n-\B_n)
\D^{-1}_n =\C_n^{-1}\A_n
\D_n^{-1}-\C_n^{-1}\B_n
\D_n^{-1}.
\end{eqnarray*}
Using (\ref{sn1}), we have
\begin{eqnarray*}
\C^{-1}_n\B_n\D^{-1}_n =
\C^{-1} \sum_{i=1}^n
\varpi_i\r_{i}\r_{i}^*\F_i^{-1}(z)
\end{eqnarray*}
and
\begin{eqnarray*}
\C^{-1}_n\A(\B_n-z\I)^{-1}
=\C^{-1}_n\A\sum_{i=1}^n
\varpi_i\r_{i}\r_{i}^*\F_i^{-1}(z)
+\alpha_n\C^{-1}_n\A\T_N
\F^{-1}(z).
\end{eqnarray*}
%
%
%
%
Then from the definition of $\De(z)$ and \eqref{mn} we have
\begin{eqnarray*}
p\De_n &=&n\E_0\varpi_1\r_{1}^*
\F_1^{-1}(z)\C^{-1}\A\r_{1}I(
\cQ_1)
\\
&&{}+\E_0\alpha_n\tr\A\T_N\F^{-1}(z)
\C^{-1}I(\cQ_1) -n\E_0\varpi_1
\r_{1}^*\F_1^{-1}(z)\C^{-1}
\r_{i}I(\cQ_1)
\\
&=&d_1+d_2+d_3+d_4,
\end{eqnarray*}
where
\begin{eqnarray*}
d_1&=&n\E_0\varpi_1I(\cQ_1)
\r_{1}^*\F_1^{-1}(z)\C^{-1}\A\r
_{1}-\E_0\varpi_1I(\cQ_1)
\E_0\tr\F_1^{-1}(z)\C^{-1}\A,
\\
d_2&=&\E_0\varpi_1I(\cQ_1)
\E_0\tr\F_1^{-1}(z)\C^{-1}\A-\E
_0\varpi_1I(\cQ_1)\E_0\tr
\F^{-1}(z)\C^{-1}\A,
\\
d_3&=&\E_0\varpi_1I(\cQ_1)
\E_0\tr\F_1^{-1}(z)\C^{-1}-n
\E_0\varpi _1I(\cQ_1)\r_{1}^*
\F_1^{-1}(z)\C^{-1} \r_{1},
\\
d_4&=&\E_0 \varpi_1I(\cQ_1)
\E_0\tr\F^{-1}(z)\C^{-1}-\E_0\varpi
_1I(\cQ_1)\E_0\tr\F_1^{-1}(z)
\C^{-1}.
\end{eqnarray*}

First, consider $d_1$. We apply \eqref{evpi} and \eqref{gai} to
represent $d_1$ as
%
\begin{eqnarray}
\label{d11}d_1&=&-n(1-z) \bigl(\varpi^{\E}_1
\bigr)^2\E_0\eta_1 \bigl(\r
_{1}^*\F_1^{-1}\C^{-1}\A
\r_{1}- n^{-1}\tr\F_1^{-1}
\C^{-1}\A \bigr)I(\cQ_1)
\\
\label{d12}&&{}+ (1-z) \bigl(\varpi^{\E}_1\bigr)^2
\E_0\xi_1 \bigl(\tr\F_1^{-1}
\C^{-1}\A-\E _0\tr\F_1^{-1}
\C^{-1}\A \bigr)I(\cQ_1)
\\
\label{d13}&&{}+n(1-z)^2\bigl(\varpi^{\E}_1
\bigr)^2 \bigl(\E_0\vp_1
\ga_1^2\r_{1}^*\F _1^{-1}
\C^{-1}\A\r_{1}\nonumber\\[-8pt]\\[-8pt]
&&\hphantom{{}+n(1-z)^2\bigl(\varpi^{\E}_1
\bigr)^2 \bigl(}{}- n^{-1}\E_0
\vp_1\ga_1^2\E_0\tr\F
_1^{-1}\C^{-1}\A \bigr)I(\cQ_1).\nonumber
\end{eqnarray}
Note that similar to \eqref{bohij} we can get that $\|\C^{-1}\|$ and
$\|\A\C^{-1}\|$ are both bounded when $z\in\cC^t$. Thus by Lemma~\ref{efef} and Lemma~\ref{leby} we obtain that
\begin{eqnarray*}
\E\bigl|\tr\F_1^{-1}(z)\C^{-1}\A-\E_0
\tr\F_1^{-1}(z)\C^{-1}\A \bigr|^2&=&
\mathrm{O}(1),
\\
\E\bigl|\r_{1}^*\F_1^{-1}(z)\C^{-1}\A
\r_{1}-n^{-1}\tr\F_1^{-1}(z)\C
^{-1}\A\bigr|^2&=&\mathrm{O}\bigl(n^{-1}\bigr).
\end{eqnarray*}
These together with \eqref{eetail}, \eqref{boetrvpi}, \eqref
{efiefil} and H\"{o}lder's inequalityimply that
\begin{eqnarray*}
(\ref{d12})=\mathrm{O}_p\bigl(n^{-1/2}\bigr)\quad  \mbox{and}\quad  (\ref{d13})=\mathrm{O}_p\bigl(\delta_n^2
\bigr).
\end{eqnarray*}
Using Lemma \eqref{le5.2}, we have
\begin{eqnarray*}
\eqref{d11}&=&-{(\frm_x-\frt-2) (1-z) \bigl(\varpi^{\E}_1
\bigr)^2y_n}\E_0 \bigl(\F_1^{-1}(z)
\bigr)_{11} \bigl(\F_1^{-1}(z)\C^{-1}
\A \bigr)_{11}I(\cQ _1)
\\
&&{}-\frac{(\frt+1)(1-z)(\varpi^{\E}_1)^2}{n}\E_0\tr\F_1^{-2}(z)\C
^{-1}\A I(\cQ_1).
\end{eqnarray*}

For $d_2$, we use \eqref{ffi} to get
\begin{eqnarray*}
d_2&=&(1-z)\E_0\varpi_1\E_0
\bigl(\vp_1I(\cQ_1) \r_1^*\F
_1^{-1}(z)\C^{-1}\A\F_1^{-1}(z)I(
\cQ_1)\r_1 \bigr)
\\
&=&\frac{(1-z)(\varpi^{\E}_1)^2}{n}\tr\F_1^{-2}(z)\C^{-1}\A
I(\cQ _1)+\mathrm{o}_p(1).
\end{eqnarray*}
Similarly, we can get
\begin{eqnarray*}
d_3&=&{(\frm_x-\frt-2) (1-z) \bigl(\varpi^{\E}_1
\bigr)^2y_n}\E_0 \bigl(\F
_1^{-1}(z) \bigr)_{11} \bigl(
\F_1^{-1}(z)\C^{-1} \bigr)_{11}I(
\cQ_1)
\\
&&{}+\frac{(\frt+1)(1-z)(\varpi^{\E}_1)^2}{n}\E_0\tr\F_1^{-2}(z)\C
^{-1}I(\cQ_1)+\mathrm{o}_p(1)
\end{eqnarray*}
and
\begin{eqnarray*}
d_4=-\frac{(1-z)(\varpi^{\E}_1)^2}{n}\tr\F_1^{-2}(z)
\C^{-1}+\mathrm{o}_p(1).
\end{eqnarray*}
Therefore combining the above four equations, we conclude that
\begin{eqnarray*}
&& \E_0\tr\D_n^{-1}I(\cQ_1)-\tr
\C_n^{-1}
\\
&&\quad  =\frac{(\frm_x-\frt-2)(1-z)(\varpi^{\E}_1)^2p}{n}\E_0 \bigl(\F _1^{-1}(z)
\bigr)_{11} \bigl(\F_1^{-1}(z)\C^{-1}(
\I-\A) \bigr)_{11}I(\cQ_1)
\\
&&\qquad {} + \frac{\frt(1-z)(\varpi^{\E}_1)^2}{n}\E_0\tr\F_1^{-2}(z)
\C ^{-1}(\I-\A) I(\cQ_1) +\mathrm{o}_p(1).
\end{eqnarray*}
By Lemmas \ref{lemma4.3}--\ref{lem7} and the fact that
\begin{eqnarray*}
\bigl\|\C^{-1}(\I-\A)\bigr\|=\bigl\llvert \alpha_n\T_N
\bigl((1-z)\E_0\vp_1I(\cQ _1)\I-z
\alpha_n\T_N \bigr)^{-1}\bigr\rrvert \leq C,
\end{eqnarray*}
we have that
\begin{eqnarray*}
\E_0\tr\D_n^{-1}-\tr\C_n^{-1}&=&
\frac{(\frm_x-\frt-2)\alpha
_n(1-z)(\varpi^{\E}_1)^2}{pn}\tr\H_0^{-1}(z)\tr\H^{-2}_0(z)
\T_N
\\
&&{}+\frac{\frt(1-z)(\varpi^{\E}_1)^2\alpha_n }{n}\tr\H_0^{-3}(z)\T _N +
\mathrm{o}_p(1),
\end{eqnarray*}
which complete the proof of this lemma.
\end{pf}

Noting the transform $ \ddot s_{n}(z)=\frac{y}{(1+z)^2} s_{n}(\frac
{z}{1+z})-\frac{1-y+z}{z(1+z)}$, $ \ddot s_{0}(z)=\frac{y}{(1+z)^2}
s_{0}(\frac{z}{1+z})-\frac{1-y+z}{z(1+z)}$ and
(3.12) in \cite{BaiS98N} we have that for $z\in\cC_t$
\begin{eqnarray*}
\biggl\| \biggl(1-\int\frac{\alpha_ny_nz(1-z)t}{[(1-z)\E\vp_1-z\alpha
_nt][(1-z)\vp_0-z\alpha_nt]}\,\mathrm{d}F^{\T_N}(t)
\biggr)^{-1}\biggr\|\leq C_\th.
\end{eqnarray*}
Thus we have $\E s_n=s_0+\mathrm{O}(n^{-1})\to s$, which combined with \eqref
{last1} gives \eqref{phi3}.

We so far have proved Lemma~\ref{leclt} under the condition that $z\in
\cC^t$. It is easy to check that the above arguments evidently work
when $z$ belongs to the bottom line due to symmetry.

When $z$ belongs to
the left vertical line of the contour, that is $z\in\cC^l=\{\Re
z=c_l-\th,\Im z\in[- \th,\th]\}$, we split $\cC^l$ into two parts
$\cC^l_1+\cC^l_2$ where
\begin{eqnarray*}
\cC^l_1=\bigl\{\Re z=c_l-
\th,n^{-1}\ep_n<|\Im z|<\th\bigr\} \quad \mbox{and}\quad
\cC^l_2=\bigl\{\Re z=c_l-\th,|\Im
z|<n^{-1}\ep_n\bigr\}
\end{eqnarray*}
with $\ep_n=n^{-\beta}$ for some $\beta\in(0,1)$.
We truncate $s_n$ at each part, that is
\begin{eqnarray*}
\hat s_n(z)= \left\{ %
\begin{array} {l@{\qquad}l} s_n(z),
& \mbox{$z\in\cC^l_1$;}
\\\noalign{\vspace*{3pt}}
s_n\bigl(\Re z+\mathrm{i}n^{-1}\ep_n\bigr), & \mbox{$z
\in\cC^l_2$.} \end{array} %
\right.
\end{eqnarray*}
Then from a similar argument in \cite{BaiS04C} we can get that the
limit of $p(\hat s_n(z)I(\cQ)-s_0)$ has the same form as Lemma~\ref
{leclt} provided. Here $\cQ=\{\|(\S_n+\alpha_n\T_N)^{-1}\|\leq C\}
\cap\{\la_1^{\S_n}>c_l-\iota\}$ with small enough $\iota>0$.
And the situation is the same if $z$ belongs to the right vertical line
of the contour due to symmetry.
We omit the details.

\section{Some basic lemmas}\label{sec5}
In this section, we give some basic lemmas which are used in the paper.
%
\begin{lemma}[(Lemma~6.1 in \cite{Zheng12C})]\label{limit}
\begin{eqnarray*}
z&=&-\frac{\dddot s(z)(\dddot s(z)+1-y)}{(\dddot s(z)+1/(1-Y))(1-Y)},\\
 \ddot s(z)&=&\frac{(\dddot s(z)+1/(1-Y))(1-Y)}{\dddot s(z)(\dddot s(z)+1)},
\\
\bigl(\dddot s(z)\bigr)'&=&-\frac{(\dddot s(z)+1/(1-Y))^2(1-Y)^2}{(1-Y)\dddot
s(z)^2+2\dddot s(z)+1-y},\\
 \int
\frac{1}{\ddot s(z)+ t}\,\mathrm{d}F_{\mathrm{mp}}^{Y}(t)&=&
\frac{\dddot s(z)}{(\dddot
s(z)+1/(1-Y))(1-Y)},
\\
\int t \bigl(\ddot s(z)+ t \bigr)^{-2}\,\mathrm{d}F_{\mathrm{mp}}^{Y}(t)&=&
\frac{(\dddot
s(z))^2}{(1-Y)\dddot s(z)^2+2\dddot s(z)+1},
\\
\ddot s'(z)=-\frac{(1-Y)\dddot s(z)^2+2\dddot s(z)+1}{(\dddot
s(z))^2(\dddot s(z)+1)^2}\bigl(\dddot s(z)
\bigr)'
&=&-{\bigl(1-Y\bigl(\dddot s(z)\bigr)^2\bigl(\dddot
s(z)+1\bigr)^2\bigr)}\dddot s^{-2}(z) \bigl(\dddot s(z)
\bigr)',
\\
\frac{2y\int\alpha t (\ddot s(z) )^3 (\ddot s(z)+
t )^{-3}\,\mathrm{d}F_{\mathrm{mp}}^{Y}(t)}{ (1-y\int (\ddot s(z)
)^2 (\ddot s(z)+ t )^{-2}\,\mathrm{d}F_{\mathrm{mp}}^{Y}(t) )^2}
&=&-\biggl(\log \frac{(1-Y)\dddot s(z)^2+2\dddot s(z)+1-y}{(1-Y)\dddot
s(z)^2+2\dddot s(z)+1} \biggr)',
\\
\frac{2Y \ddot s'(z) (\dddot s(z) )^3 (\dddot s(z)+
t )^{-3}}{ ((1-Y(\dddot s(z))^2(\dddot s(z)+1)^2)
)^2}&=&\bigl(\log \bigl(1-Y\bigl(\dddot s(z)\bigr)^2
\bigl(\dddot s(z)+1\bigr)^2\bigr) \bigr)'.
\end{eqnarray*}
\end{lemma}
%
\begin{lemma}[(Lemma~2.3 in \cite{SilversteinB95E})]\label{leAB}
Let $x,y$ be arbitrary non-negative numbers. For $\A$ and $\B$ square
matrices of the same size,
\begin{eqnarray*}
F^{\sqrt{(\A\B)(\A\B)^*}}\bigl\{(xy,\infty)\bigr\}\leq F^{\sqrt{\A\A
^*}}\bigl\{(x,\infty)
\bigr\}+F^{\sqrt{\B\B^*}}\bigl\{(y,\infty)\bigr\}.
\end{eqnarray*}
\end{lemma}
%
%
\begin{lemma}[(Lemma A.45 and Corollary A.41 in \cite{BaiS10S})]\label{A.47}
Let $\mathbf{A}$ and $\mathbf{B}$ be two $n\times n$ Hermitian
matrices. Then
\begin{eqnarray*}
L \bigl(F^{\mathbf{A}},F^{\mathbf{B}} \bigr)\leq\|\A-\B\| \quad \mbox{and}\quad
L^3 \bigl(F^{\mathbf{A}},F^{\mathbf{B}} \bigr)\leq
\frac{1}{n}\tr(\A -\B) (\A-\B)^*,
\end{eqnarray*}
where $L(\cdot,\cdot)$ denotes the L\'{e}vy distance and $\|\cdot\|$
denotes the spectral norm.
\end{lemma}
%
\begin{lemma}[(Lemma~9.1 of \cite{BaiS10S})]\label{leby} Let $\A$ be
an $n \times n$ nonrandom matrix bounded in norm by $M$, and $X =
(x_1,\dots, x_n)^*$ be a random vector of independent entries. Assume
that $\E x_i = 0$,\linebreak[4]  $\E|x_i|^2 = 1$, $\E|x_j |^4<\infty$ and
$|x_i|\leq\delta_n\sqrt{n}$ with $\delta_n\to0$ slowly. Then for
any given $2\leq l\leq b\log(n\delta_n^2)$ with some $b>1$, there
exists a constant $C$ such that
\begin{eqnarray*}
\E\bigl|X^*\A X-\tr\A\bigr|^l\leq n^l\bigl(n\delta_n^4
\bigr)^{-1}\bigl(MC\delta_n^2
\bigr)^l.
\end{eqnarray*}
\end{lemma}
%
\begin{lemma}[((1.15) of \cite{BaiS04C})]\label{le5.2}
Let $\A= (a_{ij} )_{p\times p}$ and $\mathbf{B}=
(b_{ij} )_{p\times p}$ be nonrandom matrices and $X = (x_1,\dots
, x_n)^*$ be a random vector of independent entries. Assume that $\E
x_i = 0$ and\linebreak[4]  $\E|x_i|^2 = 1$. Then we have,
%
\begin{eqnarray}
\hspace*{-20pt}\E\bigl(X^*\A X-\tr\A\bigr) \bigl(X^*\mathbf{B}X-\tr\mathbf{B}\bigr)
 = \sum
_{i=1}^p\bigl(\E|x_{i}|^4-|
\E x_{i}^2|^2-2\bigr)a_{ii}b_{ii}+
\tr\A _x\mathbf{B}^T_x+\tr\A\mathbf{B},
\end{eqnarray}
where $\A_x= (\E x_{i}^2a_{ij} )_{p\times p}$, $\mathbf
{B}_x= (\E x_{i}^2b_{ij} )_{p\times p}$ and the superscript
$^T$ is the transpose of a matrix.
\end{lemma}
%
\begin{lemma}[(Burkholder inequality)]\label{bhi}
Let $\{X_k\}$ be a complex martingale difference sequence with respect
to the increasing $\sigma$-field $\mathcal{F}_k$, and let $\E_k$
denote conditional expectation with respect to $\mathcal{F}_k$. Then
we have
\begin{enumerate}[(b)]
\item[(a)] for $p>1$,
%
\begin{eqnarray}
\E\Biggl\llvert \sum_{k=1}^n
X_k\Biggr\rrvert ^p\leq K_p \E \Biggl( \sum
_{k=1}^n|X_k|^2
\Biggr)^{p/2};
\end{eqnarray}

\item[(b)] for $p\geq2$,
%
\begin{eqnarray}
\E\Biggl\llvert \sum_{k=1}^nX_k
\Biggr\rrvert ^p\leq K_p^* \Biggl(\E \Biggl(\sum
_{k=1}^n\E_{k-1}|X_k|^2
\Biggr)^{p/2}+\E\sum_{k=1}^n|X_k|^p
\Biggr).
\end{eqnarray}
\end{enumerate}
\end{lemma}
%
\begin{lemma}\label{le5.7}
Let $\mathbf{A}$ and $\mathbf{B}$ be two $n\times n$ non-negative
definite Hermitian matrices. $\lambda_i^{\mathbf{A}}$ and $\lambda
_i^{\mathbf{B}}$ denote the $i$th smallest eigenvalue of $\mathbf{A}$
and $\mathbf{B,}$ respectively. Then we have
\begin{eqnarray*}
\la_1^{\A}\la_i^\B\leq
\la^{\A\B}_i\leq\la_n^{\A}
\la_i^\B \quad \mbox{and}\quad  \la_i^\A
\la_1^{\B}\leq\la^{\A\B}_i\leq
\la_i^\A\la_n^{\B},\qquad  i=1,\dots,n.
\end{eqnarray*}
\end{lemma}


\section*{Acknowledgements}
The authors would like to thank the referee for many constructive
comments. Zhidong Bai was partially supported by CNSF 11171057,
Fundamental Research Funds for the Central Universities, NUS Grant
R-155-000-141-112. Jiang Hu was partially supported by CNSF 11301063 and
Fundamental Research Funds for the Central Universities. Guangming Pan and Wang Zhou were partially supported by
the Ministry of Education, Singapore, under grant \# ARC 14/11.
Wang Zhou was also partially supported by a grant
R-155-000-116-112 at the National University of Singapore.



\printhistory

\end{document}